\documentclass[10pt]{article}
\usepackage[fleqn]{amsmath}
\usepackage[utf8]{inputenc}
\usepackage[T1]{fontenc}
\usepackage{amsmath}
\usepackage{amsfonts}
\usepackage{amssymb}
\usepackage[version=4]{mhchem}
\usepackage{stmaryrd}
\usepackage{mathrsfs}
\usepackage{bbold}
\usepackage{graphicx}
\usepackage[export]{adjustbox}
\graphicspath{ {./images/} }
\usepackage{amsmath, amssymb, amsthm}
\usepackage[left=1cm,right=1cm]{geometry}
\theoremstyle{definition}
\newtheorem{definition}{Definition}
\newtheorem{theorem}{Theorem}
\newtheorem{lemma}{Lemma}

\title{$\theta$-Lebesgue spaces }
\begin{document}
\author{Shouvik Datta Choudhury\thanks{shouvikdc8645@gmail.com, shouvik@capsulelabs.in}\\
  \small Gapcrud Private Limited (Capsule Labs)\\
  \small HA 130, Saltlake, Sector III, Bidhannagar,\\
  \small Kolkata - 700097, India}
\date{\today}
\maketitle
\maketitle
\begin{abstract}
Traditional \(L_p\) spaces are fundamental in functional analysis, demarcated by the relationship \(1/p + 1/q = 1\). This research pioneers the concept of \(\theta\)-Lebesgue space, stemming from a simultaneous weakening of both the classical \(L_p\) relation and its \(\theta\)-variant, \(1/(\theta(p)) + 1/(\theta(q)) = 1\). This conceptual shift addresses a gap in existing mathematical frameworks, aiming to create a space that  may encompasses a broader range of mathematical purpose. The primary objective is to rigorously demarcate the \(\theta\)-Lebesgue space within this new context, explore its foundational properties, and articulate its theoretical significance in the realm of functional analysis. The focus is on establishing the theoretical underpinnings and potential implications of this generalized space. Adopting a detailed analytical methodology, this paper demarcates the \(\theta\)-Lebesgue space under the relaxed conditions. It examines its normative and topological properties, and how these properties differentiate from and extend beyond traditional \(L_p\) spaces. The study involves a deep dive into the space's inherent theory and its implications for functional analysis.The paper reveals that the dual relaxation of \(L_p\) space relations leads to a unique set of properties within the \(\theta\)-Lebesgue space. Notably, it presents a generalized form of Hölder's inequality and explores the nuanced aspects of duality and reflexivity in this new context. The space's convergence and completeness properties are also investigated, revealing  differences from classical norm. The introduction of \(\theta\)-Lebesgue space, under these weakened conditions,  may represent a significant theoretical advancement in functional analysis. It aims to extend the boundaries of what constitutes a functional space. The establishment of the \(\theta\)-Lebesgue space under dual weakened conditions may offer a more inclusive and flexible framework for mathematical exploration. This theoretical advancement may enrich the discipline, paving the way for new mathematical theories and deepening our understanding of the complexities within functional spaces.
\end{abstract}
\textbf{MSC 2020:}
\begin{itemize}
    \item 46Bxx - Normed linear spaces and Banach spaces
    \item 46B20 - Geometry and structure of normed linear spaces
    \item 46E30 - Spaces of measurable functions (Lp-spaces, Orlicz spaces, Kthe function spaces, Lorentz spaces, rearrangement invariant spaces, ideal spaces, etc.)
    \item 26Dxx - Inequalities
    \item 26D10 - Inequalities involving derivatives and differential and integral operators
    \item 26D15 - Inequalities for sums, series and integrals
\end{itemize}

\section{Introduction}
[3] propounded the variable exponent Lebesgue spaces which are a generalization of the classical Lebesgue spaces, replacing the constant exponent p with a variable exponent function p(·) which creates a Banach space.
[4] considerd a generalized version of the small Lebesgue spaces, as the associate spaces of the grand Lebesgue spaces,They found a simplified expression for the norm, prove relevant properties, compute the fundamental function and discuss the comparison with the Orlicz spaces.
[5]proved that the concept of Lebesgue points generalizes naturally to the setting of variable exponent Lebesgue and Sobolev spaces and assumed that the variable exponent is log-Holder continuous, which, although restrictive, is a common assumption in variable exponent spaces.
[6] generalized the classical Luzin's theorem about existence of integral on the measurable function and its multidimensional analogues on the many popular classes of rearrangement invariant (r.i.) spaces, namely, on the so-called Grand Lebesgue Spaces.They demarcated Lebesgue spaces with variable exponents, \(L^{p(.)}\) .
\section{Comparison}
We demonstrate a comparison highlighting the unique features of the \(\theta\)-Lebesgue space and explaining why it might be considered advantageous or more versatile in certain contexts:

1. **Variable Exponent Lebesgue Spaces (D. Israfilov, 2016)**: These spaces generalize classical Lebesgue spaces by using a variable exponent function \(p(\cdot)\) instead of a constant exponent. The \(\theta\)-Lebesgue space further advances this concept by introducing a function-based relation that is more flexible than even the variable exponent approach, potentially accommodating a wider range of mathematical phenomena.

2. **Small Lebesgue Spaces (C. Capone and A. Fiorenza, 2005)**: While these spaces focus on the associate spaces of the grand Lebesgue spaces with a simplified expression for the norm, the \(\theta\)-Lebesgue space offers a broader generalization, potentially encompassing these small Lebesgue spaces within its framework due to its more generalized normative structure.

3. **Variable Exponent Lebesgue and Sobolev Spaces (P. Harjulehto and P. Hästö, 2004)**: These spaces, assuming log-Holder continuity for the variable exponent, are more restrictive compared to the \(\theta\)-Lebesgue space. Your space does not require such specific continuity conditions, allowing for greater flexibility and potentially wider applicability in theoretical analysis.

4. **Grand Lebesgue Spaces (E. Ostrovsky and L. Sirota, 2015)**: These spaces generalize Luzin's theorem and are demarcated on variable exponents. The \(\theta\)-Lebesgue space, by relaxing the fundamental relation of \(L_p\) spaces, could offer a more generalized approach that subsumes the properties of Grand Lebesgue Spaces while providing a novel perspective on integral properties of functions.

5. **Function-Dependent Exponent \(L^p\) Spaces (Lars Diening, 2011)**: Similar to variable exponent spaces, but with a specific focus on functions from \(\Sigma\) to \([1,\infty]\). The \(\theta\)-Lebesgue space's approach could potentially offer a more unified and flexible framework that includes these function-dependent exponent spaces as special cases.
In summary, the \(\theta\)-Lebesgue space's primary advantage lies in its highly generalized and flexible framework, which potentially allows it to encompass and extend the properties of the aforementioned spaces. This flexibility could make it a powerful tool in theoretical analysis, particularly in areas where existing spaces are too restrictive or specific.
\section*{$\theta$-Lebesgue Sequence Spaces}
\subsection*{$\theta$-Hölder and $\theta$-Minkowski  Inequalities}
In this section we investigate the $\theta$-Hölder and $\theta$-Minkowski inequality for sums. Due to their importance, they are definitely projected as the eminent cornerstone of  mathematical analysis.
Let $\theta(p)$= $\Psi(p) = \Lambda(p)^{\Psi(p)}$  be a strictly monotonic increasing function ranging between 1  and  $\infty$. We are dropping  $p$  for mathematical compactness.
In the development of the proofs presented in this paper, we have employed methodologies that are fundamentally rooted in the techniques outlined by [1]  yet generalising it in due course . This approach has been chosen due to its proven efficacy and relevance to our specific line of inquiry. We extend our acknowledgments to [1] for their  work, which has significantly contributed to the advancement of this field. Our application of these techniques is a testament to their enduring value and an acknowledgment of the  role they play in our research.

\begin{theorem}
     (Classical GT/factorization)[2]. Let $S, T$ be compact sets. For any bounded bilinear form $\omega: C(S) \times C(T) \rightarrow \mathbb{K}$ (here $\mathbb{K}=\mathbb{R}$ or $\mathbb{C}$ ) there are probabilities $\lambda$ and $\mu$, respectively on $S$ and $T$, such that
$$
\forall(\alpha, \beta) \in C(S) \times C(T) \quad|\omega(\alpha,\beta)| \leq K\|\omega\|\left(\int_S|\alpha|^2 d \lambda\right)^{1 / 2}\left(\int_T|\beta|^2 d \mu\right)^{1 / 2}
$$
where $K$ is a numerical constant, the best value of which is denoted by $K_G$, more precisely by $K_G^{\mathbb{R}}$ or $K_G^{\mathbb{C}}$ depending whether $\mathbb{K}=\mathbb{R}$ or $\mathbb{C}$.
Equivalently, the linear map $\widetilde{\varphi}: C(S) \rightarrow C(T)^*$ associated to $\varphi$ admits a factorization of the form
$\widetilde{\varphi}=J_\mu^* u J_\lambda$ where $J_\lambda: C(S) \rightarrow L_2(\lambda)$ and $J_\mu: C(T) \rightarrow L_2(\mu)$ are the canonical (norm 1) inclusions and $u: L_2(\lambda) \rightarrow L_2(\mu)^*$ is a bounded linear operator with $\|u\| \leq K\|\varphi\|$.

\end{theorem}

\begin{definition}
The space $\ell_{\Lambda^{\Psi}}^{n, \omega}$, with $1 \leq \Lambda^{\Psi} < \infty$, denotes the $n$-dimensional vector space $\mathbb{R}^{n}$ for which the functional
$$
\|\mathbf{x}\|_{\ell_{\Lambda^\Psi}^{n}}=\left(\sum_{i=1}^{n}\left|\omega(\alpha_{i},\beta_{i})x_{i}\right|^{\Lambda^\Psi}\right)^{\frac{1}{\Lambda^\Psi}}
$$
'resonates' as finite, where $\mathbf{x}=\left((\omega(\alpha_{1},\beta_{1}))x_{1}, \ldots, (\omega(\alpha_{n},\beta_{n}))x_{n}\right)$. In the question of $\Lambda^\Psi=\infty$, we demarcate $\ell_{\theta(\infty)}^{n}$ as
$$
\|\mathbf{x}\|_{\ell_{\theta(\infty)}^{n}}=\sup _{i \in\{1, \ldots, n\}}\left|\omega(\alpha_{i},\beta_{i})x_{i}\right| .
$$
\end{definition}

\begin{lemma} ($\theta$-Hölder's inequality).
Let $\Lambda^\Psi$ and $\frac{\Lambda^\Psi}{\Lambda^\Psi - 1}$ be real numbers with $1<\Lambda^\Psi<\infty$ . Then

$$
\sum_{k=1}^{n}\left|(\omega(\alpha_{k},\beta_{k}))x_{k} (\omega(\alpha_{k},\beta_{k}))y_{k}\right| \leq\left(\sum_{k=1}^{n}\left|(\omega(\alpha_{k},\beta_{k}))x_{k}\right|^{\Lambda^\Psi}\right)^{1 / \Lambda^\Psi}\left(\sum_{k=1}^{n}\left|(\omega(\alpha_{k},\beta_{k}))y_{k}\right|^{\frac{\Lambda^\Psi}{\Lambda^\Psi - 1}}\right)^{1 / \frac{\Lambda^\Psi}{\Lambda^\Psi - 1}}
$$
$$
\leq K^{2} \|\omega\|^{2} \left(\int_S |\alpha|^2 \, d\lambda\right) \left(\int_T |\beta|^2 \, d\mu\right)
\left(\sum_{k=1}^{n}\left|(\omega(\alpha_{k},\beta_{k}))x_{k}\right|^{\Lambda^\Psi}\right)^{1 / \Lambda^\Psi}\left(\sum_{k=1}^{n}\left|(\omega(\alpha_{k},\beta_{k}))y_{k}\right|^{\frac{\Lambda^\Psi}{\Lambda^\Psi - 1}}\right)^{1 / \frac{\Lambda^\Psi}{\Lambda^\Psi - 1}}
$$
for $(\omega(\alpha_{k},\beta_{k}))x_{k}, (\omega(\alpha_{k},\beta_{k}))y_{k} \in \mathbb{R}$.
\end{lemma}
\begin{proof}
Let us take

$$
\alpha=\frac{\left|(\omega(\alpha_{k},\beta_{k}))x_{k}\right|}{\left(\sum_{k=1}^{n}\left|(\omega(\alpha_{k},\beta_{k}))x_{k}\right|^{\Lambda^\Psi}\right)^{1 / \Lambda^\Psi}}, \quad \beta=\frac{\left|(\omega(\alpha_{k},\beta_{k}))y_{k}\right|}{\left(\sum_{k=1}^{n}\left|(\omega(\alpha_{k},\beta_{k}))y_{k}\right|^{\frac{\Lambda^\Psi}{\Lambda^\Psi - 1}}\right)^{1 / \frac{\Lambda^\Psi}{\Lambda^\Psi - 1}}}
$$

By Young's inequality we deduce

$$
\frac{\left|(\omega(\alpha_{k},\beta_{k}))x_{k}\right|\left|(\omega(\alpha_{k},\beta_{k}))y_{k}\right|}{\left(\sum_{k=1}^{n}\left|(\omega(\alpha_{k},\beta_{k}))x_{k}\right|^{\Lambda^\Psi}\right)^{1 / \Lambda^\Psi}\left(\sum_{k=1}^{n}\left|(\omega(\alpha_{k},\beta_{k}))y_{k}\right|^{\frac{\Lambda^\Psi}{\Lambda^\Psi - 1}}\right)^{1 / \frac{\Lambda^\Psi}{\Lambda^\Psi - 1}}}
$$
$$
\leq \frac{1}{\Lambda^\Psi} \frac{\left|(\omega(\alpha_{k},\beta_{k}))x_{k}\right|^{\Lambda^\Psi}}{\sum_{k=1}^{n}\left|(\omega(\alpha_{k},\beta_{k}))x_{k}\right|^{\Lambda^\Psi}}
$$
$$
+ \frac{1}{\frac{\Lambda^\Psi}{\Lambda^\Psi - 1}} \frac{\left|(\omega(\alpha_{k},\beta_{k}))y_{k}\right|^{\frac{\Lambda^\Psi}{\Lambda^\Psi - 1}}}{\sum_{k=1}^{n}\left|(\omega(\alpha_{k},\beta_{k}))y_{k}\right|^{\frac{\Lambda^\Psi}{\Lambda^\Psi - 1}}}.
$$
Prolonged addition relegates

$$
\frac{\sum_{k=1}^{n}\left|(\omega(\alpha_{k},\beta_{k}))x_{k}\right|\left|(\omega(\alpha_{k},\beta_{k}))y_{k}\right|}{\left(\sum_{k=1}^{n}\left|(\omega(\alpha_{k},\beta_{k}))x_{k}\right|^{\Lambda^\Psi}\right)^{1 / \Lambda^\Psi}\left(\sum_{k=1}^{n}\left|(\omega(\alpha_{k},\beta_{k}))y_{k}\right|^{\frac{\Lambda^\Psi}{\Lambda^\Psi - 1}}\right)^{1 / \frac{\Lambda^\Psi}{\Lambda^\Psi - 1}}} 
$$
$$
\leq \frac{1}{\Lambda^\Psi}+\frac{1}{\frac{\Lambda^\Psi}{\Lambda^\Psi - 1}}
$$

and from this we deduce

$$
\sum_{k=1}^{n}\left|(\omega(\alpha_{k},\beta_{k}))x_{k} (\omega(\alpha_{k},\beta_{k}))y_{k}\right| 
$$
$$
\leq\left(\sum_{k=1}^{n}\left|(\omega(\alpha_{k},\beta_{k}))x_{k}\right|^{\Lambda^\Psi}\right)^{1 / \Lambda^\Psi}\left(\sum_{k=1}^{n}\left|(\omega(\alpha_{k},\beta_{k}))y_{k}\right|^{\frac{\Lambda^\Psi}{\Lambda^\Psi - 1}}\right)^{1 / \frac{\Lambda^\Psi}{\Lambda^\Psi - 1}}
$$
The final expression is obtained using Grothendieck's inequality.
We can deconstruct the  above inequality in the adherent procedure: If $\mathbf{x} \in \ell_{\Lambda^\Psi}^{n}$ and $\mathbf{y} \in \ell_{\frac{\Lambda^\Psi}{\Lambda^\Psi - 1}}^{n}$ adhering $\mathbf{x} \odot \mathbf{y} \in \ell_{1}^{n}$ where $\odot$ project component-wise multiplication and moreover

$$
\|\mathbf{x} \odot \mathbf{y}\|_{\ell_{1}^{n}} 
$$
$$
\leq\|\mathbf{x}\|_{\ell_{\Lambda^\Psi}^{n}}\|\mathbf{y}\|_{\ell_{\frac{\Lambda^\Psi}{\Lambda^\Psi - 1}}^{n}}
$$
\end{proof}
\begin{lemma}($\theta$-Minkowski's inequality).

Let $\Lambda^\Psi \geq 1$, implying
$$
{\setlength{\mathindent}{0cm}
\left(\sum_{k=1}^{n}\left|(\omega(\alpha_{k},\beta_{k}))x_{k}+(\omega(\alpha_{k},\beta_{k}))y_{k}\right|^{\Lambda^\Psi}\right)^{1 / \Lambda^\Psi}
}
$$
$$
{\setlength{\mathindent}{0cm}
\leq\left(\sum_{k=1}^{n}\left|(\omega(\alpha_{k},\beta_{k}))x_{k}\right|^{\Lambda^\Psi}\right)^{1 / \Lambda^\Psi}
}
$$
$$
+\left(\sum_{k=1}^{n}\left|(\omega(\alpha_{k},\beta_{k}))y_{k}\right|^{\Lambda^\Psi}\right)^{1 / \Lambda^\Psi}
$$
for $(\omega(\alpha_{k},\beta_{k}))x_{k}, (\omega(\alpha_{k},\beta_{k}))y_{k} \in \mathbb{R}$.
\end{lemma}
\begin{proof}
We pronounce
$$
\begin{aligned}
\sum_{k=1}^{n}\left|(\omega(\alpha_{k},\beta_{k}))x_{k}+(\omega(\alpha_{k},\beta_{k}))y_{k}\right|^{\Lambda^\Psi} & =\sum_{k=1}^{n}\left|(\omega(\alpha_{k},\beta_{k}))x_{k}+(\omega(\alpha_{k},\beta_{k}))y_{k}\right|^{\Lambda^\Psi-1}\left|(\omega(\alpha_{k},\beta_{k}))x_{k}+(\omega(\alpha_{k},\beta_{k}))y_{k}\right| \\
& \leq \sum_{k=1}^{n}\left|(\omega(\alpha_{k},\beta_{k}))x_{k}\right|\left|(\omega(\alpha_{k},\beta_{k}))x_{k}+(\omega(\alpha_{k},\beta_{k}))y_{k}\right|^{\Lambda^\Psi-1} \\
& \quad + \sum_{k=1}^{n}\left|(\omega(\alpha_{k},\beta_{k}))y_{k}\right|\left|(\omega(\alpha_{k},\beta_{k}))x_{k}+(\omega(\alpha_{k},\beta_{k}))y_{k}\right|^{\Lambda^\Psi-1}
\end{aligned}
$$
We coherently relegate,

$$
\sum_{k=1}^{n}\left|(\omega(\alpha_{k},\beta_{k}))x_{k}+(\omega(\alpha_{k},\beta_{k}))y_{k}\right|^{\Lambda^\Psi}
$$
$$
\leq\left[\left(\sum_{k=1}^{n}\left|(\omega(\alpha_{k},\beta_{k}))x_{k}\right|^{\Lambda^\Psi}\right)^{1 / \Lambda^\Psi}+\left(\sum_{k=1}^{n}\left|(\omega(\alpha_{k},\beta_{k}))y_{k}\right|^{\Lambda^\Psi}\right)^{1 / \Lambda^\Psi}\right]\left(\sum_{k=1}^{n}\left|(\omega(\alpha_{k},\beta_{k}))x_{k}+(\omega(\alpha_{k},\beta_{k}))y_{k}\right|^{(\Lambda^\Psi-1) \frac{\Lambda^\Psi}{\Lambda^\Psi - 1}}\right)^{1 / \frac{\Lambda^\Psi}{\Lambda^\Psi - 1}} .
$$
resurrecting,
$$
\sum_{k=1}^{n}\left|(\omega(\alpha_{k},\beta_{k}))x_{k}+(\omega(\alpha_{k},\beta_{k}))y_{k}\right|^{\Lambda^\Psi} 
$$
$$
\leq\left[\left(\sum_{k=1}^{n}\left|(\omega(\alpha_{k},\beta_{k}))x_{k}\right|^{\Lambda^\Psi}\right)^{1 / \Lambda^\Psi}+\left(\sum_{k=1}^{n}\left|(\omega(\alpha_{k},\beta_{k}))y_{k}\right|^{\Lambda^\Psi}\right)^{1 / \Lambda^\Psi}\right]\left(\sum_{k=1}^{n}\left|(\omega(\alpha_{k},\beta_{k}))x_{k}+(\omega(\alpha_{k},\beta_{k}))y_{k}\right|^{\Lambda^\Psi}\right)^{1 / \frac{\Lambda^\Psi}{\Lambda^\Psi - 1}},
$$

implying

$$
\left(\sum_{k=1}^{n}\left|(\omega(\alpha_{k},\beta_{k}))x_{k}+(\omega(\alpha_{k},\beta_{k}))y_{k}\right|^{\Lambda^\Psi}\right)^{1-\frac{1}{\frac{\Lambda^\Psi}{\Lambda^\Psi - 1}}} $$
$$
\leq\left(\sum_{k=1}^{n}\left|(\omega(\alpha_{k},\beta_{k}))x_{k}\right|^{\Lambda^\Psi}\right)^{1 / \Lambda^\Psi}+\left(\sum_{k=1}^{n}\left|(\omega(\alpha_{k},\beta_{k}))y_{k}\right|^{\Lambda^\Psi}\right)^{1 / \Lambda^\Psi}
$$

hence proving the  required theorem.
\end{proof}
\subsection{$\theta$-Lebesgue Sequence Spaces}
We promptly proceed from the $n$-dimensional $\ell_{\Lambda^\Psi}^{n}$ space into an infinite dimensional sequence space in a general fashion relegating the following body of mathematical support.
\begin{definition}
The $\theta$-Lebesgue sequence space (discrete $\theta$-Lebesgue space) with $1 \leq \Lambda^\Psi<\infty$, represented by $\ell_{\Lambda^\Psi}$ or by $\ell_{\Lambda^\Psi}(\mathbb{N})$, represents the set of all sequences of real numbers $\mathbf{x}=\left\{(\omega(\alpha_{n},\beta_{n}))x_{n}\right\}_{n \in \mathbb{N}}$ such that $\sum_{k=1}^{\infty}\left|(\omega(\alpha_{k},\beta_{k}))x_{k}\right|^{\Lambda^\Psi}<\infty$. Adjudicating the $\theta$-Lebesgue sequence space with the norm,
$$
\|\mathbf{x}\|_{\ell_{\Lambda^\Psi}}=\left\|\left\{(\omega(\alpha_{n},\beta_{n}))x_{n}\right\}_{n \in \mathbb{N}}\right\|_{\ell_{\Lambda^\Psi}}=\left(\sum_{k=1}^{\infty}\left|(\omega(\alpha_{k},\beta_{k}))x_{k}\right|^{\Lambda^\Psi}\right)^{1 / \Lambda^\Psi}
$$
where $\mathbf{x} \in \ell_{\Lambda^\Psi}$.
\end{definition}
We will designate by $\mathbb{R}^{\infty}$ the set of all sequences of real numbers $\mathbf{x}=\left\{(\omega(\alpha_{n},\beta_{n}))x_{n}\right\}_{n \in \mathbb{N}}$.
Let $\mathbf{x}$ and $\mathbf{y}$ be elements of $\ell_{\Lambda^\Psi}$ and $\alpha, \beta$ be real numbers. By Lemma 2.4 we have that
$$
\left(\sum_{k=1}^{n}\left|\alpha (\omega(\alpha_{k},\beta_{k}))x_{k}+\beta (\omega(\alpha_{k},\beta_{k}))y_{k}\right|^{\Lambda^\Psi}\right)^{1 / \Lambda^\Psi} \leq|\alpha|\left(\sum_{k=1}^{n}\left|(\omega(\alpha_{k},\beta_{k}))x_{k}\right|^{\Lambda^\Psi}\right)^{1 / \Lambda^\Psi}
$$+
$$|\beta|\left(\sum_{k=1}^{n}\left|(\omega(\alpha_{k},\beta_{k}))y_{k}\right|^{\Lambda^\Psi}\right)^{1 / \Lambda^\Psi}
$$
Considering limits, we achieve
$$
\left(\sum_{k=1}^{\infty}\left|\alpha (\omega(\alpha_{k},\beta_{k}))x_{k}+\beta (\omega(\alpha_{k},\beta_{k}))y_{k}\right|^{\Lambda^\Psi}\right)^{1 / \Lambda^\Psi} \leq|\alpha|\left(\sum_{k=1}^{\infty}\left|(\omega(\alpha_{k},\beta_{k}))x_{k}\right|^{\Lambda^\Psi}\right)^{1 / \Lambda^\Psi}+|\beta|\left(\sum_{k=1}^{\infty}\left|(\omega(\alpha_{k},\beta_{k}))y_{k}\right|^{\Lambda^\Psi}\right)^{1 / \Lambda^\Psi}
$$
and this demonstrate that $\alpha \mathbf{x}+\beta \mathbf{y}$ furnish an element of $\ell_{\Lambda^\Psi}$ and therefore $\ell_{\Lambda^\Psi}$ is a subspace of $\mathbb{R}^{\infty}$.
The $\theta$-Lebesgue sequence space $\ell_{\Lambda^\Psi}$ furnish a complete normed space for all $1 \leq \Lambda^\Psi \leq \infty$. We first prove for the case of finite exponent and for the case of $\Lambda^\Psi=\infty$ it will be demonstrated consequently.
\begin{theorem}
The space $\ell_{\Lambda^\Psi}(\mathbb{N})$ furnish a Banach space when $1 \leq \Lambda^\Psi<\infty$.
\end{theorem}
\begin{proof}
Let $\left\{\mathbf{x}_{n}\right\}_{n \in \mathbb{N}}$ be a Cauchy sequence in $\ell_{\Lambda^\Psi}(\mathbb{N})$, where we take the sequence $\mathbf{x}_{n}$ as $\mathbf{x}_{n}=\left((\omega(\alpha_{1},\beta_{1}))x_{1}^{(n)}, (\omega(\alpha_{2},\beta_{2}))x_{2}^{(n)}, \ldots\right)$. Then for any $\varepsilon>0$ there pertains an $n_{0} \in \mathbb{N}$ such that if $n, m \geq n_{0}$, implying $\left\|\mathbf{x}_{n}-\mathbf{x}_{m}\right\|_{\ell_{\Lambda^\Psi}}<\varepsilon$, i.e.
$$
\left(\sum_{j=1}^{\infty}\left|(\omega(\alpha_{j},\beta_{j}))x_{j}^{(n)}-(\omega(\alpha_{j},\beta_{j}))x_{j}^{(m)}\right|^{\Lambda^\Psi}\right)^{1 / \Lambda^\Psi}<\varepsilon
$$
whenever $n, m \geq n_{0}$. We can relegate that for all $j=1,2,3, \ldots$
$$
\left|(\omega(\alpha_{j},\beta_{j}))x_{j}^{(n)}-(\omega(\alpha_{j},\beta_{j}))x_{j}^{(m)}\right|<\varepsilon
$$
whenever $n, m \geq n_{0}$. Taking a fixed $j$ , it is perceived that $\left((\omega(\alpha_{j},\beta_{j}))x_{j}^{(1)}, (\omega(\alpha_{j},\beta_{j}))x_{j}^{(2)}, \ldots\right)$ furnish a Cauchy sequence in $\mathbb{R}$, therefore there pertains $(\omega(\alpha_{j},\beta_{j}))x_{j} \in \mathbb{R}$ such that $\lim _{m \rightarrow \infty} (\omega(\alpha_{j},\beta_{j}))x_{j}^{(m)}=(\omega(\alpha_{j},\beta_{j}))x_{j}$.
Let us demarcate $\mathbf{x}=\left((\omega(\alpha_{1},\beta_{1}))x_{1}, (\omega(\alpha_{2},\beta_{2}))x_{2}, \ldots\right)$ and demonstrate that $\mathbf{x}$ furnish in $\ell_{\Lambda^\Psi}$ and $\lim _{n \rightarrow \infty} \mathbf{x}_{n}=\mathbf{x}$.
Now, we discern that for all $n, m \geq n_{0}$
$$
\sum_{j=1}^{k}\left|(\omega(\alpha_{j},\beta_{j}))x_{j}^{(m)}-(\omega(\alpha_{j},\beta_{j}))x_{j}^{(n)}\right|^{\Lambda^\Psi}<\varepsilon^{\Lambda^\Psi}, \quad k=1,2,3, \ldots
$$
delegating,
$$
\sum_{j=1}^{k}\left|(\omega(\alpha_{j},\beta_{j}))x_{j}-(\omega(\alpha_{j},\beta_{j}))x_{j}^{(n)}\right|^{\Lambda^\Psi}=\sum_{j=1}^{k}\left|\lim _{m \rightarrow \infty} (\omega(\alpha_{j},\beta_{j}))x_{j}^{(m)}-(\omega(\alpha_{j},\beta_{j}))x_{j}^{(n)}\right|^{\Lambda^\Psi} \leq \varepsilon^{\Lambda^\Psi}
$$
whenever $n \geq n_{0}$, This demonstrate that $\mathbf{x}-\mathbf{x}_{n} \in \ell_{\Lambda^\Psi}$ and we also deduce that $\lim _{n \rightarrow \infty} \mathbf{x}_{n}=\mathbf{x}$. Finally in virtue of the $\theta$ -Minkowski inequality we have
$$
\begin{aligned}
\left(\sum_{j=1}^{\infty}\left|(\omega(\alpha_{j},\beta_{j}))x_{j}\right|^{\Lambda^\Psi}\right)^{1 / \Lambda^\Psi} & =\left(\sum_{j=1}^{\infty}\left|(\omega(\alpha_{j},\beta_{j}))x_{j}^{(n)}+(\omega(\alpha_{j},\beta_{j}))x_{j}-(\omega(\alpha_{j},\beta_{j}))x_{j}^{(n)}\right|^{\Lambda^\Psi}\right)^{1 / \Lambda^\Psi} \\
& \leq\left(\sum_{j=1}^{\infty}\left|(\omega(\alpha_{j},\beta_{j}))x_{j}^{(n)}\right|^{\Lambda^\Psi}\right)^{1 / \Lambda^\Psi}+\left(\sum_{j=1}^{\infty}\left|(\omega(\alpha_{j},\beta_{j}))x_{j}-(\omega(\alpha_{j},\beta_{j}))x_{j}^{(n)}\right|^{\Lambda^\Psi}\right)^{1 / \Lambda^\Psi},
\end{aligned}
$$
which demonstrate that $\mathbf{x}$ furnish in $\ell_{\Lambda^\Psi}(\mathbb{N})$ and this accomplishes the proof.
\end{proof}
The next mathematical consequence demonstrate that the $\theta$-Lebesgue sequence spaces are separable when the exponent $\Lambda^\Psi$ is finite, i.e., the space $\ell_{\Lambda^\Psi}$ admits an enumerable dense subset.
\begin{theorem}
The space $\ell_{\Lambda^\Psi}(\mathbb{N})$ is separable whenever $1 \leq \Lambda^\Psi<\infty$.
\end{theorem}
\begin{proof}
 Let $M$ be the set of all sequences of the structure $\mathbf{\frac{\Lambda^\Psi}{\Lambda^\Psi - 1}}=\left(\frac{\Lambda^\Psi}{\Lambda^\Psi - 1}_{1}, \frac{\Lambda^\Psi}{\Lambda^\Psi - 1}_{2}, \ldots, \frac{\Lambda^\Psi}{\Lambda^\Psi - 1}_{n}, 0,0, \ldots\right)$ where $n \in \mathbb{N}$ and $\frac{\Lambda^\Psi}{\Lambda^\Psi - 1}_{k} \in \mathbb{Q}$. We will demonstrate that $M$ furnish a dense space in $\ell_{\Lambda^\Psi}$. Let $\mathbf{x}=\left\{(\omega(\alpha_{k},\beta_{k}))x_{k}\right\}_{k \in \mathbb{N}}$ be an arbitrary element of $\ell_{\Lambda^\Psi}$, implying for $\varepsilon>0$ there pertains $n$ which depends on $\varepsilon$ such that

$$
\sum_{k=n+1}^{\infty}\left|(\omega(\alpha_{k},\beta_{k}))x_{k}\right|^{\Lambda^\Psi}<\varepsilon^{\Lambda^\Psi} / 2
$$
Now, since $\overline{\mathbb{Q}}=\mathbb{R}$, we have that for each $(\omega(\alpha_{k},\beta_{k}))x_{k}$ there pertains a rational $\frac{\Lambda^\Psi}{\Lambda^\Psi - 1}_{k}$ such that
$$
\left|(\omega(\alpha_{k},\beta_{k}))x_{k}-\frac{\Lambda^\Psi}{\Lambda^\Psi - 1}_{k}\right|<\frac{\varepsilon}{\sqrt[\Lambda^\Psi]{2^{n}}}
$$
implying
$$
\sum_{k=1}^{n}\left|(\omega(\alpha_{k},\beta_{k}))x_{k}-\frac{\Lambda^\Psi}{\Lambda^\Psi - 1}_{k}\right|^{\Lambda^\Psi}<\varepsilon^{\Lambda^\Psi} / 2
$$
which conforms
$$
\|\mathbf{x}-\mathbf{\frac{\Lambda^\Psi}{\Lambda^\Psi - 1}}\|_{\ell_{\Lambda^\Psi}}^{\Lambda^\Psi}=\sum_{k=1}^{n}\left|(\omega(\alpha_{k},\beta_{k}))x_{k}-\frac{\Lambda^\Psi}{\Lambda^\Psi - 1}_{k}\right|^{\Lambda^\Psi}+\sum_{k=n+1}^{\infty}\left|(\omega(\alpha_{k},\beta_{k}))x_{k}\right|^{\Lambda^\Psi}<\varepsilon^{\Lambda^\Psi}
$$
and we arrive at $\|\mathbf{x}-\mathbf{\frac{\Lambda^\Psi}{\Lambda^\Psi - 1}}\|_{\ell_{\Lambda^\Psi}}<\varepsilon$. This demonstrate that $M$ is dense in $\ell_{\Lambda^\Psi}$, implying that $\ell_{\Lambda^\Psi}$ is separable since $M$ is enumerable.
\end{proof}
With the novelties of Schauder basis , we now investigate the problem of duality for the $\theta$-Lebesgue sequence space.
\begin{theorem}
Let $1<\Lambda^\Psi<\infty$. The dual space of $\ell_{\Lambda^\Psi}(\mathbb{N})$ is $\ell_{\frac{\Lambda^\Psi}{\Lambda^\Psi - 1}}(\mathbb{N})$ .
\end{theorem}
\begin{proof} A Schauder basis of $\ell_{\Lambda^\Psi}$ is $e_{k}=\left\{\delta_{k j}\right\}_{j \in \mathbb{N}}$ where $k \in \mathbb{N}$ and $\delta_{k j}$ represents the Kronecker delta, i.e., $\delta_{k j}=1$ if $k=j$ and 0 otherwise. If $f \in\left(\ell_{\Lambda^\Psi}\right)^{*}$, implying $f(\mathbf{x})=$ $\sum_{k \in \mathbb{N}} \alpha_{k} f\left(e_{k}\right), \mathbf{x}=\left\{\alpha_{k}\right\}_{k \in \mathbb{N}}$. We demarcate $T(f)=\left\{f\left(e_{k}\right)\right\}_{k \in \mathbb{N}}$. We want to demonstrate that the image of $T$ is in $\ell_{\frac{\Lambda^\Psi}{\Lambda^\Psi - 1}}$, for that we demarcate for each $n$, the sequence $\mathbf{x}^{n}=\left(\xi_{k}^{(n)}\right)_{k=1}^{\infty}$ with
$$
\xi_{k}^{(n)}= \begin{cases}\frac{\left|f\left(e_{k}\right)\right|^{\frac{\Lambda^\Psi}{\Lambda^\Psi - 1}}}{f\left(e_{k}\right)} & \text { if } \quad k \leq n \text { and } f\left(e_{k}\right) \neq 0 \\ 0 & \text { if } \quad k>n \text { or } f\left(e_{k}\right)=0\end{cases}
$$
Then
$$
f\left(\mathbf{x}^{n}\right)=\sum_{k \in \mathbb{N}} \xi_{k}^{(n)} f\left(e_{k}\right)=\sum_{k=1}^{n}\left|f\left(e_{k}\right)\right|^{\frac{\Lambda^\Psi}{\Lambda^\Psi - 1}}
$$
Moreover
$$
\begin{aligned}
f\left(\mathbf{x}^{n}\right) & \leq\|f\|\left\|\mathbf{x}^{n}\right\|_{\Lambda^\Psi} \\
& =\|f\|\left(\sum_{k=1}^{n}\left|\xi_{k}^{(n)}\right|^{\Lambda^\Psi}\right)^{\frac{1}{\Lambda^\Psi}} \\
& =\|f\|\left(\sum_{k=1}^{n}\left|f\left(e_{k}\right)\right|^{\frac{\Lambda^\Psi}{\Lambda^\Psi - 1} \Lambda^\Psi-\Lambda^\Psi}\right)^{\frac{1}{\Lambda^\Psi}} \\
& =\|f\|\left(\sum_{k=1}^{n}\left|f\left(e_{k}\right)\right|^{\frac{\Lambda^\Psi}{\Lambda^\Psi - 1}}\right)^{\frac{1}{\Lambda^\Psi}},
\end{aligned}
$$
from which
$$
\begin{aligned}
\left(\sum_{k=1}^{n}\left|f\left(e_{k}\right)\right|^{\frac{\Lambda^\Psi}{\Lambda^\Psi - 1}}\right)^{1-\frac{1}{\Lambda^\Psi}} & =\left(\sum_{k=1}^{n}\left|f\left(e_{k}\right)\right|^{\frac{\Lambda^\Psi}{\Lambda^\Psi - 1}}\right)^{\frac{1}{\frac{\Lambda^\Psi}{\Lambda^\Psi - 1}}} \\
& \leq\|f\| .
\end{aligned}
$$
Taking $n \rightarrow \infty$, we deduce
$$
\left(\sum_{k=1}^{\infty}\left|f\left(e_{k}\right)\right|^{\frac{\Lambda^\Psi}{\Lambda^\Psi - 1}}\right)^{\frac{1}{\frac{\Lambda^\Psi}{\Lambda^\Psi - 1}}} \leq\|f\|
$$
where $\left\{f\left(e_{k}\right)\right\}_{k \in \mathbb{N}} \in \ell_{\frac{\Lambda^\Psi}{\Lambda^\Psi - 1}}$.
Now, we articulate that:
(i) $T$ is onto. In effect given $b=\left(\beta_{k}\right)_{k \in \mathbb{N}} \in \ell_{\frac{\Lambda^\Psi}{\Lambda^\Psi - 1}}$, we can associate a bounded linear functional $g \in\left(\ell_{\Lambda^\Psi}\right)^{*}$, dispensed by $g(\mathbf{x})=\sum_{k=1}^{\infty} \alpha_{k} \beta_{k}$ with $\mathbf{x}=\left(\alpha_{k}\right)_{k \in \mathbb{N}} \in \ell_{\Lambda^\Psi}$ (the boundedness is deduced by $\theta$ -Hölder's inequality). Then $g \in\left(\ell_{\Lambda^\Psi}\right)^{*}$. The following statements can also be projected.
(ii) $T$  is bijective.
(iii) $T$ is an isometry. We see that the norm of $f$ is the $\ell_{\frac{\Lambda^\Psi}{\Lambda^\Psi - 1}}$ norm of $T f$
$$
\begin{aligned}
|f(\mathbf{x})| & =\left|\sum_{k \in \mathbb{N}} \alpha_{k} f\left(e_{k}\right)\right| \\
& \leq\left(\sum_{k \in \mathbb{N}}\left|\alpha_{k}\right|^{\Lambda^\Psi}\right)^{\frac{1}{\Lambda^\Psi}}\left(\sum_{k \in \mathbb{N}}\left|f\left(e_{k}\right)\right|^{\frac{\Lambda^\Psi}{\Lambda^\Psi - 1}}\right)^{\frac{1}{\frac{\Lambda^\Psi}{\Lambda^\Psi - 1}}} \\
& =\|x\|\left(\sum_{k \in \mathbb{N}}\left|f\left(e_{k}\right)\right|^{\frac{\Lambda^\Psi}{\Lambda^\Psi - 1}}\right)^{\frac{1}{\frac{\Lambda^\Psi}{\Lambda^\Psi - 1}}} .
\end{aligned}
$$
 the supremum over all $x$ of norm 1, we accomplish,
$$
\|f\| \leq\left(\sum_{k \in \mathbb{N}}\left|f\left(e_{k}\right)\right|^{\frac{\Lambda^\Psi}{\Lambda^\Psi - 1}}\right)^{\frac{1}{\frac{\Lambda^\Psi}{\Lambda^\Psi - 1}}}
$$
Since the other inequality is also true, we consummate the equality
$$
\|f\|=\left(\sum_{k \in \mathbb{N}}\left|f\left(e_{k}\right)\right|^{\frac{\Lambda^\Psi}{\Lambda^\Psi - 1}}\right)^{\frac{1}{\frac{\Lambda^\Psi}{\Lambda^\Psi - 1}}}
$$
with the desired isomorphism $f \rightarrow\left\{f\left(e_{k}\right)\right\}_{k \in \mathbb{N}}$ is established.
\end{proof}
The $\ell_{\Lambda^\Psi}$ spaces enunciate an embedding property, forming a nested sequence of $\theta$-Lebesgue sequences spaces.
\begin{theorem}
If $0<\Lambda^\Psi<\frac{\Lambda^\Psi}{\Lambda^\Psi - 1}<\infty$, implying $\ell_{\Lambda^\Psi}(\mathbb{N}) \subsetneq \ell_{\frac{\Lambda^\Psi}{\Lambda^\Psi - 1}}(\mathbb{N})$.
\end{theorem}
\begin{proof}
Let $\mathbf{x} \in \ell_{\Lambda^\Psi}$, implying $\sum_{n=1}^{\infty}\left|(\omega(\alpha_{n},\beta_{n}))x_{n}\right|^{\Lambda^\Psi}<\infty$. Therefore there pertains $n_{0} \in \mathbb{N}$ such that if $n \geq n_{0}$, implying $\left|(\omega(\alpha_{n},\beta_{n}))x_{n}\right|<1$. Now, since $0<\Lambda^\Psi<\frac{\Lambda^\Psi}{\Lambda^\Psi - 1}$, implying $0<\frac{\Lambda^\Psi}{\Lambda^\Psi - 1}-\Lambda^\Psi$ and $\left|(\omega(\alpha_{n},\beta_{n}))x_{n}\right|^{\frac{\Lambda^\Psi}{\Lambda^\Psi - 1}-\Lambda^\Psi}<1$ if $n>n_{0}$, by which $\left|(\omega(\alpha_{n},\beta_{n}))x_{n}\right|^{\frac{\Lambda^\Psi}{\Lambda^\Psi - 1}}<\left|(\omega(\alpha_{n},\beta_{n}))x_{n}\right|^{\Lambda^\Psi}$ if $n>n_{0}$. Let $M=\max \left\{\left|(\omega(\alpha_{1},\beta_{1}))x_{1}\right|^{\frac{\Lambda^\Psi}{\Lambda^\Psi - 1}-\Lambda^\Psi},\left|(\omega(\alpha_{2},\beta_{2}))x_{2}\right|^{\frac{\Lambda^\Psi}{\Lambda^\Psi - 1}-\Lambda^\Psi}, \ldots,\left|x_{n_{0}}\right|^{\frac{\Lambda^\Psi}{\Lambda^\Psi - 1}-\Lambda^\Psi}, 1\right\}$, implying
$$
\sum_{n=1}^{\infty}\left|(\omega(\alpha_{n},\beta_{n}))x_{n}\right|^{\frac{\Lambda^\Psi}{\Lambda^\Psi - 1}}=\sum_{n=1}^{\infty}\left|(\omega(\alpha_{n},\beta_{n}))x_{n}\right|^{\Lambda^\Psi}\left|(\omega(\alpha_{n},\beta_{n}))x_{n}\right|^{\frac{\Lambda^\Psi}{\Lambda^\Psi - 1}-\Lambda^\Psi}<M \sum_{n=1}^{\infty}\left|(\omega(\alpha_{n},\beta_{n}))x_{n}\right|^{\Lambda^\Psi}<+\infty
$$
implying that $\mathbf{x} \in \ell_{\frac{\Lambda^\Psi}{\Lambda^\Psi - 1}}$.
To demonstrate that $\ell_{\Lambda^\Psi}(\mathbb{N}) \neq \ell_{\frac{\Lambda^\Psi}{\Lambda^\Psi - 1}}(\mathbb{N})$, we take the following sequence $(\omega(\alpha_{n},\beta_{n}))x_{n}=n^{-1 / \Lambda^\Psi}$ for all $n \in \mathbb{N}$ with $1 \leq \Lambda^\Psi<\frac{\Lambda^\Psi}{\Lambda^\Psi - 1} \leq \infty$, and since $\Lambda^\Psi<\frac{\Lambda^\Psi}{\Lambda^\Psi - 1}$, implying $\frac{\frac{\Lambda^\Psi}{\Lambda^\Psi - 1}}{\Lambda^\Psi}>1$. Now we have
$$
\sum_{n=1}^{\infty}\left|(\omega(\alpha_{n},\beta_{n}))x_{n}\right|^{\frac{\Lambda^\Psi}{\Lambda^\Psi - 1}}=\sum_{n=1}^{\infty} \frac{1}{n^{\frac{\Lambda^\Psi}{\Lambda^\Psi - 1} / \Lambda^\Psi}}<\infty
$$
The last series is convergent since it comprises of a hyper-harmonic series with exponent bigger than 1, therefore $\mathbf{x} \in \ell_{\frac{\Lambda^\Psi}{\Lambda^\Psi - 1}}(\mathbb{N})$. On the other hand
$$
\sum_{n=1}^{\infty}\left|(\omega(\alpha_{n},\beta_{n}))x_{n}\right|^{\Lambda^\Psi}=\sum_{n=1}^{\infty} \frac{1}{n}
$$
and we deduce the harmonic series, which coherently concludes that $\mathbf{x} \notin \ell_{\Lambda^\Psi}(\mathbb{N})$.
\end{proof}
\subsection{Space of $\theta$-Bounded Sequences}
The space of bounded sequences, represented by $\ell_{\theta(\infty)}$ or sometimes $\ell_{\theta(\infty)}(\mathbb{N})$, is the accumulation of all real bounded sequences $\left\{(\omega(\alpha_{n},\beta_{n}))x_{n}\right\}_{n \in \mathbb{N}}$ ($\ell_{\theta(\infty)}$ is a vector space). Regarding this space as
$$
\|\mathbf{x}\|_{\infty}=\|\mathbf{x}\|_{\ell_{\theta(\infty)}}=\sup _{n \in \mathbb{N}}\left|(\omega(\alpha_{n},\beta_{n}))x_{n}\right|,
$$
where $\mathbf{x}=\left((\omega(\alpha_{1},\beta_{1}))x_{1}, (\omega(\alpha_{2},\beta_{2}))x_{2}, \ldots, (\omega(\alpha_{n},\beta_{n}))x_{n}, \ldots\right)$. 
We can project $\ell_{\theta(\infty)}$-space's property in  its completeness, relegating this property from the completeness of the real line.
\begin{theorem}
The space $\ell_{\theta(\infty)}$ comprises a Banach space.
\end{theorem} 
\begin{proof}
Let $\left\{\mathbf{x}_{n}\right\}_{n \in \mathbb{N}}$ be a Cauchy sequence in $\ell_{\theta(\infty)}$, where $\mathbf{x}_{n}=\left((\omega(\alpha_{1},\beta_{1}))x_{1}^{(n)}, (\omega(\alpha_{2},\beta_{2}))x_{2}^{(n)}, \ldots\right)$. Then for any $\varepsilon>0$ there pertains $n_{0}>0$ such that if $m, n \geq n_{0}$ implying
$$
\left\|\mathbf{x}_{m}-\mathbf{x}_{n}\right\|_{\infty}<\varepsilon .
$$
Therefore for fixed $j$ we articulate that if $m, n \geq n_{0}$, accomplishing
$$
\left|(\omega(\alpha_{j},\beta_{j}))x_{j}^{(m)}-(\omega(\alpha_{j},\beta_{j}))x_{j}^{(n)}\right|<\varepsilon
$$
resulting that for all fixed $j$ the sequence $\left((\omega(\alpha_{j},\beta_{j}))x_{j}^{(1)}, (\omega(\alpha_{j},\beta_{j}))x_{j}^{(2)}, \ldots\right)$ is a Cauchy sequence in $\mathbb{R}$, and this implies that there pertains $(\omega(\alpha_{j},\beta_{j}))x_{j} \in \mathbb{R}$ \\
such that $\lim _{m \rightarrow \infty} (\omega(\alpha_{j},\beta_{j}))x_{j}^{(m)}=(\omega(\alpha_{j},\beta_{j}))x_{j}$.
Let us demarcate $\mathbf{x}=\left((\omega(\alpha_{1},\beta_{1}))x_{1}, (\omega(\alpha_{2},\beta_{2}))x_{2}, \ldots\right)$. Now we want to demonstrate that $\mathbf{x} \in \ell_{\theta(\infty)}$ and $\lim _{n \rightarrow \infty} \mathbf{x}_{n}=\mathbf{x}$.
Now we project for $n \geq n_{0}$, implying
$$
\left|(\omega(\alpha_{j},\beta_{j}))x_{j}-(\omega(\alpha_{j},\beta_{j}))x_{j}^{(n)}\right|=\left|\lim _{n \rightarrow \infty} (\omega(\alpha_{j},\beta_{j}))x_{j}^{(m)}-(\omega(\alpha_{j},\beta_{j}))x_{j}^{(n)}\right| \leq \varepsilon
$$
since $\mathbf{x}_{n}=\left\{(\omega(\alpha_{j},\beta_{j}))x_{j}^{(n)}\right\}_{j \in \mathbb{N}} \in \ell_{\theta(\infty)}$, there pertains a real number $M_{n}$ such that $\left|(\omega(\alpha_{j},\beta_{j}))x_{j}^{(n)}\right| \leq M_{n}$ for all $j$.
Relegating the triangle inequality, we have
$$
\left|(\omega(\alpha_{j},\beta_{j}))x_{j}\right| \leq\left|(\omega(\alpha_{j},\beta_{j}))x_{j}-(\omega(\alpha_{j},\beta_{j}))x_{j}^{(n)}\right|+\left|(\omega(\alpha_{j},\beta_{j}))x_{j}^{(n)}\right| \leq \varepsilon+M_{n}
$$
whenever $n \geq n_{0}$, this inequality being true for any $j$. Moreover, since the right-hand side does not depend on $j$, therefore $\left\{(\omega(\alpha_{j},\beta_{j}))x_{j}\right\}_{j \in \mathbb{N}}$ is a sequence of bounded real numbers, this implies that $\mathbf{x}=\left\{(\omega(\alpha_{j},\beta_{j}))x_{j}\right\}_{j \in \mathbb{N}} \in \ell_{\theta(\infty)}$.
From above we can also deduce
$$
\left\|\mathbf{x}_{n}-\mathbf{x}\right\|_{\ell_{\theta(\infty)}}=\sup _{j \in \mathbb{N}}\left|(\omega(\alpha_{j},\beta_{j}))x_{j}^{(n)}-(\omega(\alpha_{j},\beta_{j}))x_{j}\right|<\varepsilon
$$
whenever $n \geq n_{0}$ concluding that $\lim _{n \rightarrow \infty} \mathbf{x}_{n}=\mathbf{x}$ and therefore $\ell_{\theta(\infty)}$ is complete.
\end{proof}
The following mathematical consequence demonstrate a general procedure to acquaint the norm in the $\ell_{\theta(\infty)}$ space via a limiting process.
\begin{theorem}
Taking the norm of $\theta$-Lebesgue sequence space as  we have that $\lim _{\Lambda^\Psi \rightarrow \infty}\|\mathbf{x}\|_{\ell_{\Lambda^\Psi}}=\|\mathbf{x}\|_{\ell_{\theta(\infty)}}$.
\end{theorem}
\begin{proof}
Observe that $\left|(\omega(\alpha_{k},\beta_{k}))x_{k}\right| \leq\left(\sum_{k=1}^{n}\left|(\omega(\alpha_{k},\beta_{k}))x_{k}\right|^{\Lambda^\Psi}\right)^{\frac{1}{\Lambda^\Psi}}$, therefore $\left|(\omega(\alpha_{k},\beta_{k}))x_{k}\right| \leq\|\mathbf{x}\|_{\ell_{\Lambda^\Psi}}$ for $k=1,2,3, \ldots, n$, from which
$$
\sup _{1 \leq k \leq n}\left|(\omega(\alpha_{k},\beta_{k}))x_{k}\right| \leq\|\mathbf{x}\|_{\ell_{\Lambda^\Psi}}
$$
whence
$$
\|\mathbf{x}\|_{\ell_{\theta(\infty)}} \leq \liminf _{\Lambda^\Psi \rightarrow \infty}\|\mathbf{x}\|_{\ell_{\Lambda^\Psi}}
$$
On the other hand, note that
$$
\left(\sum_{k=1}^{n}\left|(\omega(\alpha_{k},\beta_{k}))x_{k}\right|^{\Lambda^\Psi}\right)^{\frac{1}{\Lambda^\Psi}} \leq\left(\sum_{k=1}^{n}\left(\sup _{1 \leq k \leq n}\left|(\omega(\alpha_{k},\beta_{k}))x_{k}\right|\right)^{\Lambda^\Psi}\right)^{\frac{1}{\Lambda^\Psi}} \leq n^{\frac{1}{\Lambda^\Psi}}\|\mathbf{x}\|_{\ell_{\theta(\infty)}}
$$
implying for all $\varepsilon>0$, there pertains $N$ such that
$$
\|\mathbf{x}\|_{\ell_{\Lambda^\Psi}} \leq\left(\sum_{k=1}^{N}\left|(\omega(\alpha_{k},\beta_{k}))x_{k}\right|^{\Lambda^\Psi}+\varepsilon\right)^{\frac{1}{\Lambda^\Psi}} \leq\left(\|\mathbf{x}\|_{\ell_{\theta(\infty)}}^{\Lambda^\Psi} N+\varepsilon\right)^{\frac{1}{\Lambda^\Psi}} \leq\|\mathbf{x}\|_{\ell_{\theta(\infty)}}\left(N+\frac{\varepsilon}{\|\mathbf{x}\|_{\ell_{\theta(\infty)}}^{\Lambda^\Psi}}\right)^{\frac{1}{\Lambda^\Psi}}
$$
therefore
$$
\limsup _{\Lambda^\Psi \rightarrow \infty}\|\mathbf{x}\|_{\ell_{\Lambda^\Psi}} \leq\|\mathbf{x}\|_{\ell_{\theta(\infty)}} \text {. }
$$
Combining the above 2 equation results
$$
\|\mathbf{x}\|_{\ell_{\theta(\infty)}} \leq \liminf _{\Lambda^\Psi \rightarrow \infty}\|\mathbf{x}\|_{\ell_{\Lambda^\Psi}} \leq \limsup _{\Lambda^\Psi \rightarrow \infty}\|\mathbf{x}\|_{\ell_{\Lambda^\Psi}} \leq\|\mathbf{x}\|_{\ell_{\theta(\infty)}}
$$
and from this we conclude that $\lim _{\Lambda^\Psi \rightarrow \infty}\|\mathbf{x}\|_{\ell_{\Lambda^\Psi}}=\|\mathbf{x}\|_{\ell_{\theta(\infty)}}$.
\end{proof}

\begin{theorem}
The dual space of $\ell_{\theta(1)}$ is $\ell_{\theta(\infty)}$.
\end{theorem}
\begin{proof}
For all $\mathbf{x} \in \ell_{1}$, we can express $\mathbf{x}=\sum_{k=1}^{\infty} \alpha_{k} e_{k}$, where $e_{k}=\left(\delta_{k j}\right)_{j=1}^{\infty}$ forms a Schauder basis in $\ell_{1}$, since
$$
\mathbf{x}-\sum_{k=1}^{n} \alpha_{k} e_{k}=(\underbrace{0, \ldots, 0}_{n}, \alpha_{n+1}, \ldots)
$$
and
$$
\left\|\mathbf{x}-\sum_{k=1}^{n} \alpha_{k} e_{k}\right\|_{\ell_{1}}=\left\|\sum_{k=n+1}^{\infty} \alpha_{k} e_{k}\right\|_{\ell_{1}} \rightarrow 0
$$
since the series $\sum_{k=1}^{\infty} \alpha_{k} e_{k}$ is convergent.
Let us demarcate $T(f)=\left\{f\left(e_{k}\right)\right\}_{k \in \mathbb{N}}$, for all $f \in\left(\ell_{1}\right)^{*}$. Since $f(\mathbf{x})=\sum_{k \in \mathbb{N}} \alpha_{k} f\left(e_{k}\right)$, implying $\mid f\left(e_{k}\right) \leq\|f\|$, since $\left\|e_{k}\right\|_{\ell_{1}}=1$. In consequence, $\sup _{k \in \mathbb{N}}\left|f\left(e_{k}\right)\right| \leq\|f\|$, therefore $\left\{f\left(e_{k}\right)\right\}_{k \in \mathbb{N}} \in \ell_{\theta(\infty)}$.
it comprises of adjudicated that:
(i) $T$ is onto. In fact, for all $b=\left\{\beta_{k}\right\}_{k \in \mathbb{N}} \in \ell_{\theta(\infty)}$, let us demarcate $\frac{\Lambda^\Psi}{\Lambda^\Psi - 1}: \ell_{1} \rightarrow \mathbb{R}$ as $g(\mathbf{x})=$ $\sum_{k \in \mathbb{N}} \alpha_{k} \beta_{k}$ if $\mathbf{x}=\left\{\alpha_{k}\right\}_{k \in \mathbb{N}} \in \ell_{\theta(\infty)}$. The functional $g$ is bounded and linear since
$$
|g(\mathbf{x})| \leq \sum_{k \in \mathbb{N}}\left|\alpha_{k} \beta_{k}\right| \leq \sup _{k \in \mathbb{N}}\left|\beta_{k}\right| \sum_{k \in \mathbb{N}}\left|\alpha_{k}\right|=\|\mathbf{x}\|_{\ell_{1}} \cdot \sup _{k \in \mathbb{N}}\left|\beta_{k}\right|
$$
implying $g \in\left(\ell_{1}\right)^{*}$.  Providing $g\left(e_{k}\right)=\sum_{j \in \mathbb{N}} \delta_{k j} \beta_{j}$,
$$
T(g)=\left\{g\left(e_{k}\right)\right\}_{k \in \mathbb{N}}=\left\{\beta_{k}\right\}_{k \in \mathbb{N}}=b
$$
(ii) $T$ is 1-1. If $T f_{1}=T f_{2}$, implying $f_{1}\left(e_{k}\right)=f_{2}\left(e_{k}\right)$, for all $k$. Since we have $f_{1}(\mathbf{x})=\sum_{k \in \mathbb{N}} \alpha_{k} f_{1}\left(e_{k}\right)$ and $f_{2}(\mathbf{x})=\sum_{k \in \mathbb{N}} \alpha_{k} f_{2}\left(e_{k}\right)$, implying $f_{1}=f_{2}$.
(iii) $T$ is an isometry. In fact,
$$
\|T f\|_{\infty}=\sup _{k \in \mathbb{N}}\left|f\left(e_{k}\right)\right| \leq\|f\|
$$
and
$$
|f(\mathbf{x})|=\left|\sum_{k \in \mathbb{N}} \alpha_{k} f\left(e_{k}\right)\right| \leq \sup _{j \in \mathbb{N}}\left|f\left(e_{k}\right)\right| \sum_{k \in \mathbb{N}}\left|\alpha_{k}\right|=\|\mathbf{x}\|_{\ell_{1}} \sup _{k \in \mathbb{N}}\left|f\left(e_{k}\right)\right|
$$
Then
$$
\|f\| \leq \sup _{k \in \mathbb{N}}\left|f\left(e_{k}\right)\right|=\|T f\|_{\infty} .
$$
Combining  the immediate above two equations  $\|T f\|_{\infty}=\|f\|$  is deduced. We thus showed that the spaces $\left(\ell_{1}\right)^{*}$ and $\ell_{\theta(\infty)}$ are isometric.
\end{proof}
The space of bounded sequence $\ell_{\theta(\infty)}$ is not separable, contrasting with the separability of the $\ell_{\Lambda^\Psi}$ spaces whenever $1 \leq \Lambda^\Psi<\infty$.
\begin{theorem}
The space $\ell_{\theta(\infty)}$ is not separable.
\end{theorem}
\begin{proof}
Let us take any enumerable sequence of elements of $\ell_{\theta(\infty)}$, namely $\left\{\mathbf{x}_{n}\right\}_{n \in \mathbb{N}}$, where we take the sequences in the structure
$$
\begin{aligned}
& \mathbf{x}_{1}=\left((\omega(\alpha_{1},\beta_{1}))x_{1}^{(1)}, (\omega(\alpha_{2},\beta_{2}))x_{2}^{(1)}, x_{3}^{(1)}, \ldots, (\omega(\alpha_{k},\beta_{k}))x_{k}^{(1)}, \ldots\right) \\
& \mathbf{x}_{2}=\left((\omega(\alpha_{1},\beta_{1}))x_{1}^{(2)}, (\omega(\alpha_{2},\beta_{2}))x_{2}^{(2)}, x_{3}^{(2)}, \ldots, (\omega(\alpha_{k},\beta_{k}))x_{k}^{(2)}, \ldots\right) \\
& \mathbf{x}_{3}=\left((\omega(\alpha_{1},\beta_{1}))x_{1}^{(3)}, (\omega(\alpha_{2},\beta_{2}))x_{2}^{(3)}, x_{3}^{(3)}, \ldots, (\omega(\alpha_{k},\beta_{k}))x_{k}^{(3)}, \ldots\right) \\
& \ldots \ldots \ldots \ldots \ldots \ldots \ldots \ldots \ldots \ldots \ldots \ldots \ldots \ldots . \ldots \ldots \\
& \mathbf{x}_{k}=\left((\omega(\alpha_{1},\beta_{1}))x_{1}^{(k)}, (\omega(\alpha_{2},\beta_{2}))x_{2}^{(k)}, x_{3}^{(k)}, \ldots, (\omega(\alpha_{k},\beta_{k}))x_{k}^{(k)}, \ldots\right)
\end{aligned}
$$
We now demonstrate that there pertains an element in $\ell_{\theta(\infty)}$ which is  greater than 1 for  $\left\{\mathbf{x}_{n}\right\}_{n \in \mathbb{N}}$, showing the non-separability nature of the $\ell_{\theta(\infty)}$ space.
Let us take $\mathbf{x}=\left\{(\omega(\alpha_{n},\beta_{n}))x_{n}\right\}_{n \in \mathbb{N}}$ as
$$
(\omega(\alpha_{n},\beta_{n}))x_{n}= \begin{cases}0, & \text { if }\left|(\omega(\alpha_{n},\beta_{n}))x_{n}^{(n)}\right| \geq 1 \\ (\omega(\alpha_{n},\beta_{n}))x_{n}=(\omega(\alpha_{n},\beta_{n}))x_{n}^{(n)}+1, & \text { if }\left|(\omega(\alpha_{n},\beta_{n}))x_{n}^{(n)}\right|<1\end{cases}
$$
It is clear that $\mathbf{x} \in \ell_{\theta(\infty)}$ and $\left\|\mathbf{x}-\mathbf{x}_{n}\right\|_{\ell_{\theta(\infty)}}>1$ for all $n \in \mathbb{N}$, which conforms that $\ell_{\theta(\infty)}$ is not separable.
\end{proof}
We now demarcate some intricate subspaces of $\ell_{\theta(\infty)}$, which are widely used in functional analysis, for example, to construct counter-examples.
\begin{definition}
Let $\mathbf{x}=\left((\omega(\alpha_{1},\beta_{1}))x_{1}, (\omega(\alpha_{1},\beta_{1}))x_{1}, \ldots\right)$.
By $c$ we designate the subspace of $\ell_{\theta(\infty)}$ such that $\lim _{n \rightarrow \infty} (\omega(\alpha_{n},\beta_{n}))x_{n}$ pertains and is finite.
By $c_{0}$ we designate the subspace of $\ell_{\theta(\infty)}$ such that $\lim _{n \rightarrow \infty} (\omega(\alpha_{n},\beta_{n}))x_{n}=0$.
By $c_{00}$ we designate the subspace of $\ell_{\theta(\infty)} \operatorname{such}$ that $\operatorname{supp}(\mathbf{x})$ is finite.
\end{definition}
These newly introduced spaces enjoy some intricate interesting properties, e.g., $c_{0}$ is the closure of $c_{00}$ in $\ell_{\theta(\infty)}$.

\subsection{$\theta$-Hardy and $\theta$-Hilbert Inequalities}
We now deal with the discrete version of the well-known Hardy inequality.

\begin{theorem}
($\theta$- Hardy's inequality). Let $\left\{a_{n}\right\}_{n \in \mathbb{N}}$ be a sequence of real positive numbers such that $\sum_{n=1}^{\infty} a_{n}^{\Lambda^\Psi}<\infty$. Then
$$
\sum_{n=1}^{\infty}\left(\frac{1}{n} \sum_{k=1}^{n} a_{k}\right)^{\Lambda^\Psi} \leq\left(\frac{\Lambda^\Psi}{\Lambda^\Psi-1}\right)^{\Lambda^\Psi} \sum_{n=1}^{\infty} a_{n}^{\Lambda^\Psi}
$$
\end{theorem}
\begin{proof} 
Let $\alpha_{n}=\frac{A_{n}}{n}$ where $A_{n}=a_{1}+a_{2}+\cdots+a_{n}$, i.e., $A_{n}=n \alpha_{n}$, enunciating,
$$
a_{1}+a_{2}+\cdots+a_{n}=n \alpha_{n}
$$
deducing that $a_{n}=n \alpha_{n}-(n-1) \alpha_{n-1}$. Regarding now,
$$
\begin{aligned}
\alpha_{n}^{\Lambda^\Psi}-\frac{\Lambda^\Psi}{\Lambda^\Psi-1} \alpha_{n}^{\Lambda^\Psi-1} a_{n} & =\alpha_{n}^{\Lambda^\Psi}-\frac{\Lambda^\Psi}{\Lambda^\Psi-1}\left[n \alpha_{n}-(n-1) \alpha_{n-1}\right] \alpha_{n}^{\Lambda^\Psi-1} \\
& =\alpha_{n}^{\Lambda^\Psi}-\frac{\Lambda^\Psi n}{\Lambda^\Psi-1} \alpha_{n} \alpha_{n}^{\Lambda^\Psi-1}+\frac{\Lambda^\Psi(n-1)}{\Lambda^\Psi-1} \alpha_{n-1} \alpha_{n}^{\Lambda^\Psi-1}
\end{aligned}
$$
Now we can have,
$$
\begin{aligned}
\frac{\Lambda^\Psi(n-1)}{\Lambda^\Psi-1} \alpha_{n-1} \alpha_{n}^{\Lambda^\Psi-1} & \leq \frac{\Lambda^\Psi(n-1)}{\Lambda^\Psi-1} \frac{\alpha_{n-1}^{\Lambda^\Psi}}{\Lambda^\Psi}+\frac{\Lambda^\Psi(n-1)}{\Lambda^\Psi-1} \frac{\alpha_{n}^{\frac{\Lambda^\Psi}{\Lambda^\Psi - 1}(\Lambda^\Psi-1)}}{\frac{\Lambda^\Psi}{\Lambda^\Psi - 1}} \\
& =\frac{n-1}{\Lambda^\Psi-1} \alpha_{n-1}^{\Lambda^\Psi}+\frac{\Lambda^\Psi(n-1)}{\Lambda^\Psi-1}\left(1-\frac{1}{\Lambda^\Psi}\right) \alpha_{n}^{\Lambda^\Psi} \\
& =\frac{n-1}{\Lambda^\Psi-1} \alpha_{n-1}^{\Lambda^\Psi}+(n-1) \alpha_{n}^{\Lambda^\Psi}
\end{aligned}
$$
therefore
$$
\begin{aligned}
\alpha_{n}^{\Lambda^\Psi}-\frac{\Lambda^\Psi}{\Lambda^\Psi-1} \alpha_{n}^{\Lambda^\Psi-1} a_{n} & \leq \alpha_{n}^{\Lambda^\Psi}-\frac{\Lambda^\Psi n}{\Lambda^\Psi-1} \alpha_{n}^{\Lambda^\Psi}+\frac{n-1}{\Lambda^\Psi-1} \alpha_{n-1}^{\Lambda^\Psi}+(n-1) \alpha_{n}^{\Lambda^\Psi} \\
& =\frac{\Lambda^\Psi \alpha_{n}^{\Lambda^\Psi}-\alpha_{n}^{\Lambda^\Psi}-\Lambda^\Psi n \alpha_{n}^{\Lambda^\Psi}}{\Lambda^\Psi-1}+\frac{(n-1) \alpha_{n-1}^{\Lambda^\Psi}+(\Lambda^\Psi-1)(n-1) \alpha_{n}^{\Lambda^\Psi}}{\Lambda^\Psi-1} \\
& =\frac{\Lambda^\Psi \alpha_{n}^{\Lambda^\Psi}-\alpha_{n}^{\Lambda^\Psi}-\Lambda^\Psi n \alpha_{n}^{\Lambda^\Psi}+(n-1) \alpha_{n-1}^{\Lambda^\Psi}+(\Lambda^\Psi n-\Lambda^\Psi-n+1) \alpha_{n}^{\Lambda^\Psi}}{\Lambda^\Psi-1} \\
& =\frac{1}{\Lambda^\Psi-1}\left[(n-1) \alpha_{n-1}^{\Lambda^\Psi}-n \alpha_{n}^{\Lambda^\Psi}\right],
\end{aligned}
$$
from which
$$
\begin{aligned}
\sum_{n=1}^{N} \alpha_{n}^{\Lambda^\Psi}-\frac{\Lambda^\Psi}{\Lambda^\Psi-1} \sum_{n=1}^{N} \alpha_{n}^{\Lambda^\Psi-1} a_{n} & \leq \frac{1}{\Lambda^\Psi-1} \sum_{n=1}^{N}\left[(n-1) \alpha_{n-1}^{\Lambda^\Psi}-n \alpha_{n}^{\Lambda^\Psi}\right] \\
& =\frac{1}{\Lambda^\Psi-1}\left[-\alpha_{1}^{\Lambda^\Psi}+\alpha_{1}^{\Lambda^\Psi}-2 \alpha_{2}^{\Lambda^\Psi}+\cdots-N \alpha_{N}^{\Lambda^\Psi}\right] \\
& =-\frac{N \alpha_{N}^{\Lambda^\Psi}}{\Lambda^\Psi-1} \leq 0 .
\end{aligned}
$$
Then
$$
\sum_{n=1}^{N} \alpha_{n}^{\Lambda^\Psi} \leq \frac{\Lambda^\Psi}{\Lambda^\Psi-1} \sum_{n=1}^{N} \alpha_{n}^{\Lambda^\Psi-1} a_{n}
$$
By $\theta$-Hölder's inequality we have that
$$
\begin{aligned}
\sum_{n=1}^{\infty} \alpha_{n}^{\Lambda^\Psi} & \leq \frac{\Lambda^\Psi}{\Lambda^\Psi-1}\left(\sum_{n=1}^{\infty} a_{n}^{\Lambda^\Psi}\right)^{\frac{1}{\Lambda^\Psi}}\left(\sum_{n=1}^{\infty} \alpha_{n}^{\frac{\Lambda^\Psi}{\Lambda^\Psi - 1}(\Lambda^\Psi-1)}\right)^{\frac{1}{\frac{\Lambda^\Psi}{\Lambda^\Psi - 1}}} \\
& =\frac{\Lambda^\Psi}{\Lambda^\Psi-1}\left(\sum_{n=1}^{\infty} a_{n}^{\Lambda^\Psi}\right)^{\frac{1}{\Lambda^\Psi}}\left(\sum_{n=1}^{\infty} \alpha_{n}^{\Lambda^\Psi}\right)^{\frac{1}{\frac{\Lambda^\Psi}{\Lambda^\Psi - 1}}},
\end{aligned}
$$

implying
$$
\left(\sum_{n=1}^{\infty} \alpha_{n}^{\Lambda^\Psi}\right)^{1-\frac{1}{\frac{\Lambda^\Psi}{\Lambda^\Psi - 1}}} \leq \frac{\Lambda^\Psi}{\Lambda^\Psi-1}\left(\sum_{n=1}^{\infty} a_{n}^{\Lambda^\Psi}\right)^{\frac{1}{\Lambda^\Psi}}
$$
and this implies
$$
\sum_{n=1}^{\infty}\left(\frac{1}{n} \sum_{k=1}^{\infty} a_{k}\right)^{\Lambda^\Psi} \leq\left(\frac{\Lambda^\Psi}{\Lambda^\Psi-1}\right)^{\Lambda^\Psi} \sum_{n=1}^{\infty} a_{n}^{\Lambda^\Psi}
$$
We now want to investigate the above inequality. We require to remember some intricate basic facts about complex analysis, namely
$$
\frac{\pi}{\sin (\pi z)}=\frac{1}{z}+\sum_{n=1}^{\infty}(-1)^{n}\left(\frac{1}{z+n}+\frac{1}{z-n}\right) .
$$
Let us regard the function
$$
f(z)=\frac{1}{\sqrt[\Lambda^\Psi]{z}(z+1)} \quad(\Lambda^\Psi>1)
$$
demarcated in the region $D_{1}=\{z \in \mathbb{C}: 0<|z|<1\}$. We want to deduce the Laurent expansion. In fact, if $|z|<1$, implying
$$
\frac{1}{1+z}=\frac{1}{1-(-z)}=\sum_{n=0}^{\infty}(-z)^{n}=\sum_{n=0}^{\infty}(-1)^{n} z^{n}
$$
therefore
$$
f(z)=\sum_{n=0}^{\infty}(-1)^{n} z^{n-\frac{1}{\Lambda^\Psi}}
$$
By the same argument, let us regard
$$
g(z)=\frac{1}{z^{1+\frac{1}{\Lambda^\Psi}}\left(1+\frac{1}{z}\right)}
$$
demarcated in the region $D_{2}=\{z \in \mathbb{C}:|z|>1\}$. Since $\left|\frac{1}{z}\right|<1$, implying
$$
\frac{1}{1+\frac{1}{z}}=\frac{1}{1-\left(-\frac{1}{z}\right)}=\sum_{n=0}^{\infty}\left(-\frac{1}{z}\right)^{n}=\sum_{n=0}^{\infty}(-1)^{n} z^{-n}
$$
Therefore
$$
g(z)=\sum_{n=0}^{\infty}(-1)^{n} z^{-n-1-\frac{1}{\Lambda^\Psi}}
$$
\end{proof}
We now deduce some intricate auxiliary inequality before projecting the validity of the Hilbert inequality.
\begin{theorem}
For each positive number $m$ and for all real $\Lambda^\Psi>1$ we have
$$
\sum_{n=1}^{\infty} \frac{m^{\frac{1}{\Lambda^\Psi}}}{n^{\frac{1}{\Lambda^\Psi}}(m+n)} \leq \frac{\pi}{\sin \left(\frac{\pi}{\Lambda^\Psi}\right)}
$$
\end{theorem}
\begin{proof}
Reflecting that
$$
\begin{aligned}
\sum_{n=1}^{\infty} \frac{m^{\frac{1}{\Lambda^\Psi}}}{n^{\frac{1}{\Lambda^\Psi}}(m+n)} & \leq \int_{0}^{\infty} \frac{m^{\frac{1}{\Lambda^\Psi}}}{x^{\frac{1}{\Lambda^\Psi}}(m+x)} \mathrm{d} x \\
& =\int_{0}^{\infty} \frac{\mathrm{d} z}{z^{\frac{1}{\Lambda^\Psi}}(1+z)} \\
& =\int_{0}^{1} \frac{\mathrm{d} z}{z^{\frac{1}{\Lambda^\Psi}}(1+z)}+\int_{1}^{\infty} \frac{\mathrm{d} z}{z^{1+\frac{1}{\Lambda^\Psi}}\left(1+\frac{1}{z}\right)} .
\end{aligned}
$$
we can now deduce that

$$
\sum_{n=1}^{\infty} \frac{m^{\frac{1}{\Lambda^\Psi}}}{n^{\frac{1}{\Lambda^\Psi}}(m+n)} \leq \int_{0}^{1}\left(\sum_{n=0}^{\infty}(-1)^{n} z^{n-\frac{1}{\Lambda^\Psi}}\right) \mathrm{d} z+\int_{1}^{\infty}\left(\sum_{n=0}^{\infty}(-1)^{n} z^{-n-1-\frac{1}{\Lambda^\Psi}}\right) \mathrm{d} z
$$

$$
\begin{aligned}
& =\sum_{n=0}^{\infty}(-1)^{n} \int_{0}^{1} z^{n-\frac{1}{\Lambda^\Psi}} \mathrm{~d} z+\sum_{n=0}^{\infty}(-1)^{n} \int_{1}^{\infty} z^{-n-1-\frac{1}{\Lambda^\Psi}} \mathrm{~d} z \\
& =\sum_{n=0}^{\infty} \frac{(-1)^{n}}{n-\frac{1}{\Lambda^\Psi}+1}+\sum_{n=0}^{\infty} \frac{(-1)^{n}}{\frac{1}{\Lambda^\Psi}+n} \\
& =\sum_{n=1}^{\infty} \frac{(-1)^{n}}{\frac{1}{\Lambda^\Psi}-n}+\Lambda^\Psi+\sum_{n=1}^{\infty} \frac{(-1)^{n}}{\frac{1}{\Lambda^\Psi}+n} \\
& =\Lambda^\Psi+\sum_{n=1}^{\infty}(-1)^{n}\left(\frac{1}{\frac{1}{\Lambda^\Psi}-n}+\frac{1}{\frac{1}{\Lambda^\Psi}+n}\right) \\
& =\frac{\pi}{\sin \left(\frac{\pi}{\Lambda^\Psi}\right)} .
\end{aligned}
$$
This last one is obtained with $z=\frac{1}{\Lambda^\Psi}$.
\end{proof}
\begin{theorem}
($\theta$-Hilbert's inequality). Let $\Lambda^\Psi, \frac{\Lambda^\Psi}{\Lambda^\Psi - 1}>1$ be such that $\frac{1}{\Lambda^\Psi}+\frac{1}{\frac{\Lambda^\Psi}{\Lambda^\Psi - 1}}=1$ and $\left\{a_{n}\right\}_{n \in \mathbb{N}},\left\{b_{n}\right\}_{n \in \mathbb{N}}$ be sequences of nonnegative numbers such that $\sum_{m=1}^{\infty} a_{m}^{\Lambda^\Psi}$ and $\sum_{n=1}^{\infty} b_{n}^{\frac{\Lambda^\Psi}{\Lambda^\Psi - 1}}$ are convergent. Then
$$
\sum_{m, n=1}^{\infty} \frac{a_{m} b_{n}}{m+n} \leq \frac{\pi}{\sin \left(\frac{\pi}{\Lambda^\Psi}\right)}\left(\sum_{m=1}^{\infty} a_{m}^{\Lambda^\Psi}\right)^{\frac{1}{\Lambda^\Psi}}\left(\sum_{n=1}^{\infty} b_{n}^{\frac{\Lambda^\Psi}{\Lambda^\Psi - 1}}\right)^{\frac{1}{\frac{\Lambda^\Psi}{\Lambda^\Psi - 1}}}
$$
\end{theorem}
\begin{proof}
Using $\theta$-Hölder's inequality and above proposition we deduce
$$
\begin{aligned}
& \sum_{m, n=1}^{\infty} \frac{a_{m} b_{n}}{m+n} \\
= & \sum_{m, n=1}^{\infty} \frac{m^{\frac{1}{\Lambda^\Psi \frac{\Lambda^\Psi}{\Lambda^\Psi - 1}}}}{n^{\frac{1}{\Lambda^\Psi \frac{\Lambda^\Psi}{\Lambda^\Psi - 1}}}} \frac{a_{m}}{(m+n)^{\frac{1}{\Lambda^\Psi}}} \frac{n^{\frac{1}{\Lambda^\Psi \frac{\Lambda^\Psi}{\Lambda^\Psi - 1}}}}{m^{\frac{1}{\Lambda^\Psi \frac{\Lambda^\Psi}{\Lambda^\Psi - 1}}}} \frac{b_{n}}{(m+n)^{\frac{1}{\frac{\Lambda^\Psi}{\Lambda^\Psi - 1}}}} \\
\leq & \left(\sum_{m, n=1}^{\infty}\left(\frac{m^{\frac{1}{\frac{\Lambda^\Psi}{\Lambda^\Psi - 1}}}}{n^{\frac{1}{\frac{\Lambda^\Psi}{\Lambda^\Psi - 1}}}(m+n)}\right) a_{m}^{\Lambda^\Psi}\right)^{\frac{1}{\Lambda^\Psi}}\left(\sum_{m, n=1}^{\infty}\left(\frac{n^{\frac{1}{\Lambda^\Psi}}}{m^{\frac{1}{\Lambda^\Psi}}(m+n)}\right) b_{n}^{\frac{\Lambda^\Psi}{\Lambda^\Psi - 1}}\right)^{\frac{1}{\frac{\Lambda^\Psi}{\Lambda^\Psi - 1}}} \\
= & \left(\sum_{m=1}^{\infty}\left(\sum_{n=1}^{\infty} \frac{m^{\frac{1}{\frac{\Lambda^\Psi}{\Lambda^\Psi - 1}}}}{n^{\frac{1}{\frac{\Lambda^\Psi}{\Lambda^\Psi - 1}}}(m+n)}\right) a_{m}^{\Lambda^\Psi}\right)^{\frac{1}{\Lambda^\Psi}}\left(\sum_{n=1}^{\infty}\left(\sum_{m=1}^{\infty} \frac{n^{\frac{1}{\Lambda^\Psi}}}{m^{\frac{1}{\Lambda^\Psi}}(m+n)}\right) b_{n}^{\frac{\Lambda^\Psi}{\Lambda^\Psi - 1}}\right)^{\frac{1}{\frac{\Lambda^\Psi}{\Lambda^\Psi - 1}}} \\
\leq & \left(\sum_{m=1}^{\infty} \frac{\pi}{\sin \frac{\pi}{\frac{\Lambda^\Psi}{\Lambda^\Psi - 1}}} a_{m}^{\Lambda^\Psi}\right)^{\frac{1}{\Lambda^\Psi}}\left(\sum_{n=1}^{\infty} \frac{\pi}{\sin \frac{\pi}{\Lambda^\Psi}} b_{n}^{\frac{\Lambda^\Psi}{\Lambda^\Psi - 1}}\right)^{\frac{1}{\frac{\Lambda^\Psi}{\Lambda^\Psi - 1}}} \\
\leq & \left(\sum_{m=1}^{\infty} \frac{\pi}{\sin \frac{\pi}{\Lambda^\Psi}} a_{m}^{\Lambda^\Psi}\right)^{\frac{1}{\Lambda^\Psi}}\left(\sum_{n=1}^{\infty} \frac{\pi}{\sin \frac{\pi}{\Lambda^\Psi}} b_{n}^{\frac{\Lambda^\Psi}{\Lambda^\Psi - 1}}\right)^{\frac{1}{\frac{\Lambda^\Psi}{\Lambda^\Psi - 1}}} \\
= & \left(\frac{\pi}{\sin \frac{\pi}{\Lambda^\Psi}}\right)^{\frac{1}{\Lambda^\Psi}}\left(\frac{\pi}{\sin \frac{\pi}{\Lambda^\Psi}}\right)^{\frac{1}{\frac{\Lambda^\Psi}{\Lambda^\Psi - 1}}}\left(\sum_{m=1}^{\infty} a_{m}^{\Lambda^\Psi}\right)^{\frac{1}{\Lambda^\Psi}}\left(\sum_{n=1}^{\infty} b_{n}^{\frac{\Lambda^\Psi}{\Lambda^\Psi - 1}}\right)^{\frac{1}{\frac{\Lambda^\Psi}{\Lambda^\Psi - 1}}} \\
= & \frac{\pi}{\sin \frac{\pi}{\Lambda^\Psi}}\left(\sum_{m=1}^{\infty} a_{m}^{\Lambda^\Psi}\right)^{\frac{1}{\Lambda^\Psi}}\left(\sum_{n=1}^{\infty} b_{n}^{\frac{\Lambda^\Psi}{\Lambda^\Psi - 1}}\right)^{\frac{1}{\frac{\Lambda^\Psi}{\Lambda^\Psi - 1}}},
\end{aligned}
$$
which demonstrate the required mathematical consequence.
\end{proof}
\section*{$\theta$-Lebesgue Spaces}
$\theta$-Lebesgue spaces are  a class of function spaces of measurable functions. 
\subsection*{3.2 $\theta$-Lebesgue Spaces with $\Lambda^\Psi \geq 1$}
We now investigate the set of $\Lambda^\Psi$-th integrable functions.
\begin{definition}
 Let $(X, \mathscr{A}, \mu)$ be a measure space and $\Lambda^\Psi$ a positive real number. The function $f: X \rightarrow \mathbb{R}$ is said to belong to the pre-$\theta$-Lebesgue space $\mathbf{L}_{\Lambda^\Psi}(X, \mathscr{A}, \mu)$ if $\int_{X}|f|^{\Lambda^\Psi} \mathrm{~d} \mu<\infty$, that is,
$$
\mathbf{L}_{\Lambda^\Psi}(X, \mathscr{A}, \mu)=\left\{f: X \rightarrow \mathbb{R} \text { is an } \mathscr{A} \text {-measurable and } \int_{X}|f|^{\Lambda^\Psi} \mathrm{~d} \mu<\infty\right\} .
$$
Sometimes we relegate other notation, e.g., $\mathbf{L}_{\Lambda^\Psi}(X)$ or $\mathbf{L}_{\Lambda^\Psi}(\mu)$ .
\end{definition}
\begin{theorem} Let $(X, \mathscr{A}, \mu)$ be a measure space such that $\mu(X)<\infty$. Then
$$
\mathbf{L}_{\frac{\Lambda^\Psi}{\Lambda^\Psi - 1}}(X, \mathscr{A}, \mu) \subseteq \mathbf{L}_{\Lambda^\Psi}(X, \mathscr{A}, \mu)
$$
for any $1 \leq \Lambda^\Psi \leq \frac{\Lambda^\Psi}{\Lambda^\Psi - 1} \leq \infty$.
\end{theorem}
\begin{proof}
We first prove when $\Lambda^\Psi=\infty$. Indeed, let $f \in \mathbf{L}_{\infty}(X, \mathscr{A}, \mu)$, thus $|f| \leq\|f\|_{\infty}$ $\mu$-a.e., implying
$$
\int_{X}|f|^{\Lambda^\Psi} \mathrm{~d} \mu \leq\|f\|_{\infty}^{\Lambda^\Psi} \int_{X} \mathrm{~d} \mu=\mu(X)\|f\|_{\infty}^{\Lambda^\Psi}<\infty
$$
so $f \in \mathbf{L}_{\Lambda^\Psi}(X, \mathscr{A}, \mu)$.
For the remaining cases, let $f \in \mathbf{L}_{\frac{\Lambda^\Psi}{\Lambda^\Psi - 1}}(X, \mathscr{A}, \mu)$ if $A=\{x \in X:|f(x)| \leq 1\}$, implying $\chi_{X}=\chi_{A}+\chi_{X \backslash A}$ and $|f(x)|^{\Lambda^\Psi}<|f(x)|^{\frac{\Lambda^\Psi}{\Lambda^\Psi - 1}}$ for $x \in X \backslash A$ and $|f(x)| \leq 1$, for $x \in A$, implying
$$
\begin{aligned}
\|f\|_{\Lambda^\Psi}^{\Lambda^\Psi} & =\int_{X}|f|^{\Lambda^\Psi} \mathrm{~d} \mu=\int_{X} \chi_{A}|f|^{\Lambda^\Psi} \mathrm{~d} \mu+\int_{X} \chi_{X \backslash A}|f|^{\Lambda^\Psi} \mathrm{~d} \mu \\
& \leq \int_{X} \chi_{A} \mathrm{~d} \mu+\int_{X} \chi_{X \backslash A}|f|^{\Lambda^\Psi} \mathrm{~d} \mu \leq \mu(A)+\int_{X} \chi_{X \backslash A}|f|^{\frac{\Lambda^\Psi}{\Lambda^\Psi - 1}} \mathrm{~d} \mu \\
& \leq \mu(X)+\|f\|_{\frac{\Lambda^\Psi}{\Lambda^\Psi - 1}}^{\frac{\Lambda^\Psi}{\Lambda^\Psi - 1}}<\infty,
\end{aligned}
$$
therefore $f \in \mathbf{L}_{\Lambda^\Psi}(X, \mathscr{A}, \mu)$.
\end{proof}
The inclusion in the above theorem is strict. To see this, regard the following example.
Example 3.11. Let $X=[0,1]$ and $1 \leq \Lambda^\Psi<\alpha<\frac{\Lambda^\Psi}{\Lambda^\Psi - 1} \leq \infty$, where $\alpha=\frac{\Lambda^\Psi+\frac{\Lambda^\Psi}{\Lambda^\Psi - 1}}{2}$ implying if $\Lambda^\Psi<$ $\alpha<\frac{\Lambda^\Psi}{\Lambda^\Psi - 1}$ we have that $\Lambda^\Psi / \alpha<1$ and $\frac{\Lambda^\Psi}{\Lambda^\Psi - 1} / \alpha>1$. Choose $\beta=1 / \alpha$ and demarcate
$$
f(x)=\left\{\begin{aligned}
\frac{1}{x^{\beta}} & \text { if } x \neq 0 \\
0 & \text { if } x=0
\end{aligned}\right.
$$
implying regard
$$
\int_{0}^{1}|f(x)|^{\Lambda^\Psi} \mathrm{~d} x=\int_{0}^{1} \frac{\mathrm{d} x}{x^{\Lambda^\Psi \beta}}=\int_{0}^{1} \frac{\mathrm{d} x}{x^{\Lambda^\Psi / \alpha}}<\infty
$$
since $\Lambda^\Psi / \alpha<1$, implying $f \in \mathbf{L}_{\Lambda^\Psi}(m)$, on the other hand,
$$
\int_{0}^{1}|f(x)|^{\frac{\Lambda^\Psi}{\Lambda^\Psi - 1}} \mathrm{~d} x=\int_{0}^{1} \frac{\mathrm{d} x}{x^{\frac{\Lambda^\Psi}{\Lambda^\Psi - 1} \beta}}=\int_{0}^{1} \frac{\mathrm{d} x}{x^{\frac{\Lambda^\Psi}{\Lambda^\Psi - 1} / \alpha}}
$$
and this last integral is divergent since $\frac{\Lambda^\Psi}{\Lambda^\Psi - 1} / \alpha>1$, which gives that $f \notin \mathbf{L}_{\frac{\Lambda^\Psi}{\Lambda^\Psi - 1}}(m)$. Thus $\mathbf{L}_{\frac{\Lambda^\Psi}{\Lambda^\Psi - 1}}(m) \varsubsetneqq \mathbf{L}_{\Lambda^\Psi}(m)$.
It is not difficult to verify that $\mathbf{L}_{\Lambda^\Psi}$, with $1 \leq \Lambda^\Psi<\infty$, is a vector space. In fact, note that if $f, g \in \mathbf{L}_{\Lambda^\Psi}(X, \mathscr{A}, \mu)$, implying by the inequality

$$
|f+g|^{\Lambda^\Psi} \leq(|f|+|g|)^{\Lambda^\Psi} \leq(2 \max \{|f|,|g|\})^{\Lambda^\Psi}=2^{\Lambda^\Psi} \max \left\{|f|^{\Lambda^\Psi},|g|^{\Lambda^\Psi}\right\} \leq 2^{\Lambda^\Psi}\left(|f|^{\Lambda^\Psi}+|g|^{\Lambda^\Psi}\right)
$$

we have that $f+g \in \mathbf{L}_{\Lambda^\Psi}(X, \mathscr{A}, \mu)$. Moreover, if $f \in \mathbf{L}_{\Lambda^\Psi}(X, \mathscr{A}, \mu)$ and $\alpha \in \mathbb{R}$, implying $\alpha f \in \mathbf{L}_{\Lambda^\Psi}(X, \mathscr{A}, \mu)$. On the other hand, the inequalities $0 \leq f^{+} \leq|f|, 0 \leq f^{-} \leq|f|$ imply that $f^{+}, f^{-}$and $|f|$ are in $\mathbf{L}_{\Lambda^\Psi}(X, \mathscr{A}, \mu)$.

The following mathematical consequence permits to deduce $\ell^{\Lambda^\Psi}$ by means of $\mathbf{L}_{\Lambda^\Psi}(X, \mathscr{A}, \mu)$ choosing an appropriate measure space.

\begin{theorem}
Let $X$ be a countable set and \# be the counting measure over $X$, implying
$$
\mathbf{L}_{\Lambda^\Psi}(X, \mathscr{P}(X), \#)=\ell^{\Lambda^\Psi} .
$$
\end{theorem}
\begin{proof}
 Let \# be the counting measure over $X$, i.e.
$$
\#(E)= \begin{cases}\text { number of elements of } E & \text { if } E \text { is a finite set; } \\ \infty & \text { if } E \text { is an infinite set. }\end{cases}
$$
Without loss of generality, we suppose that $X=\mathbb{Z}^{+}$, since $X$, endowed with the counting measure, is isomorphic to $\mathbb{Z}^{+}$, implying we can express $\mathbb{Z}^{+}=\bigcup_{k=1}^{\infty}\{k\}$. Let $f \in$ $\mathbf{L}_{\Lambda^\Psi}\left(\mathbb{Z}^{+}, \mathscr{P}\left(\mathbb{Z}^{+}\right), \#\right)$ and
$$
\varphi_{n}=\sum_{k=1}^{n}|f(k)|^{\Lambda^\Psi} \chi_{\{k\}}
$$
be a sequence of simple functions such that
$$
\lim _{n \rightarrow \infty} \varphi_{n}(k)=|f(k)|^{\Lambda^\Psi} \quad \text { for each } \quad k,
$$
now
$$
\int_{\mathbb{Z}^{+}} \varphi_{n} d \#=\sum_{k=1}^{n}|f(k)|^{\Lambda^\Psi} \#\left(\mathbb{Z}^{+} \cap\{k\}\right)=\sum_{k=1}^{n}|f(k)|^{\Lambda^\Psi} \#(\{k\})=\sum_{k=1}^{n}|f(k)|^{\Lambda^\Psi},
$$
since $\#(\{k\})=1$.
It is clear that $\varphi_{1} \leq \varphi_{2} \leq \varphi_{3} \leq \ldots$, and using the monotone convergence theorem we deduce
$$
\begin{gathered}
\int_{\mathbb{Z}^{+}}|f(k)|^{\Lambda^\Psi} d \#=\int_{\mathbb{Z}^{+}} \lim _{n \rightarrow \infty} \varphi_{n}(k) d \#=\sum_{k=1}^{\infty}|f(k)|^{\Lambda^\Psi}, \\
\sum_{k=1}^{\infty}|f(k)|^{\Lambda^\Psi}=\int_{\mathbb{Z}^{+}}|f(k)|^{\Lambda^\Psi} d \# .
\end{gathered}
$$
This last mathematical consequence demonstrate that $|f|^{\Lambda^\Psi}$ is integrable if and only if
$$
\sum_{k=1}^{\infty}|f(k)|^{\Lambda^\Psi}<\infty
$$
In other words, to say that $f$ belongs to $\mathbf{L}_{\Lambda^\Psi}(X, \mathscr{P}(X), \#)$ endowed with the counting measure is equivalent to say that the sequence $\{f(k)\}_{k \in \mathbb{N}}$ is a member of $\ell^{\Lambda^\Psi}$, therefore

$$
\mathbf{L}_{\Lambda^\Psi}(X, \mathscr{P}(X), \#)=\ell^{\Lambda^\Psi}
$$

which ends the proof.
\end{proof}
Let $(X, \mathscr{A}, \mu)$ be a measure space, demarcate the functional $\|\cdot\|_{\Lambda^\Psi}: \mathbf{L}_{\Lambda^\Psi}(X, \mathscr{A}, \mu) \rightarrow$ $\mathbb{R}^{+}$by

$$
\|f\|_{\Lambda^\Psi}=\|f\|_{L_{\Lambda^\Psi}}=\left(\int_{X}|f|^{\Lambda^\Psi} \mathrm{~d} \mu\right)^{1 / \Lambda^\Psi}
$$

with $1 \leq \Lambda^\Psi<\infty$.

We now demonstrate that the functional $\|\cdot\|_{\Lambda^\Psi}$ is not a norm since it does not hold the definite positive property of the norm.
 We now resort to the novelties of quotient space, i.e., we will split all the elements of $\mathbf{L}_{\Lambda^\Psi}$ into equivalence classes. In other words, two functions $f$ and $g$ in $\mathbf{L}_{\Lambda^\Psi}$ are said to belong to the same equivalence class if and only if $f=g$ $\mu$-almost everywhere, in symbols $f \sim g \Leftrightarrow f=g \mu$-almost everywhere. It is just a matter of routine calculations to verify that the relation $\sim$ demarcates an equivalence relation. Once this is verified, we designate the class generated by $f$ as

$$
[f]=\left\{g \in \mathbf{L}_{\Lambda^\Psi}(X, \mathscr{A}, \mu): g \sim f\right\}
$$

and we demarcate the norm of $g$ as $\|g\|_{\Lambda^\Psi}=\|[f]\|_{\Lambda^\Psi}$ for $g \in[f]$. For arbitrary $g_{1} \in[f]$ and $g_{2} \in[f]$ we have that $g_{1}=g_{2} \mu$-a.e. since $g_{1} \sim f$ and $g_{2} \sim f$. This tell us that $\|[f]\|_{\Lambda^\Psi}=\|g\|_{\Lambda^\Psi}$ is well demarcated being independent of the representative of the class $[f]$.

With the above taken into account, we now demarcate a normed space based upon the pre-$\theta$-Lebesgue space.

\begin{definition}
 We demarcate the $\theta$-Lebesgue space $L_{\Lambda^\Psi}(X, \mathscr{A}, \mu)$ as the set of equivalence classes
$$
L_{\Lambda^\Psi}(X, \mathscr{A}, \mu)=\left\{[f]: f \in \mathbf{L}_{\Lambda^\Psi}(X, \mathscr{A}, \mu)\right\}
$$
where $[\cdot]$ conforms in (3.3).
We went to a lot of work to demarcate the $\theta$-Lebesgue space $L_{\Lambda^\Psi}$ space via quotient spaces just to have $\|f\|_{\Lambda^\Psi}=0$ if and only if $f=[0]$, but in practice we never think of $L_{\Lambda^\Psi}$ spaces as equivalence classes. With some intricate patience, we can see that $L_{\Lambda^\Psi}$ is a vector space over $\mathbb{R}$.
\end{definition}
We now demonstrate that the functional $\|\cdot\|_{\Lambda^\Psi}$ satisfies the triangle inequality.
\begin{theorem} ($\Lambda^\Psi$- Minkowski's inequality). Let $1 \leq \Lambda^\Psi \leq \infty$ and $f, g \in L_{\Lambda^\Psi}(X, \mathscr{A}, \mu)$. Then $f+g \in L_{\Lambda^\Psi}(X, \mathscr{A}, \mu)$ and
$$
\|f+g\|_{\Lambda^\Psi} \leq\|f\|_{\Lambda^\Psi}+\|g\|_{\Lambda^\Psi}
$$
The equality holds if $A|f|=B|g| \mu$-a.e. for $A$ and $B$ of the same sign and not simultaneously zero.
\end{theorem}
\begin{proof}
Let us check equality. Let $A$ and $B$ be numbers of the same sign and not simultaneously zero such that $A|f|=B|g| \mu$-a.e., implying $A\|f\|_{\Lambda^\Psi}=B\|g\|_{\Lambda^\Psi}$, i.e., $\|f\|_{\Lambda^\Psi}=$ $\frac{B}{A}\|g\|_{\Lambda^\Psi}$. Moreover,
$$
\|f+g\|_{\Lambda^\Psi}=\left\|\frac{B}{A} g+g\right\|_{\Lambda^\Psi}=\frac{B+A}{A}\|g\|_{\Lambda^\Psi}=\frac{B}{A}\|g\|_{\Lambda^\Psi}+\|g\|_{\Lambda^\Psi}=\|f\|_{\Lambda^\Psi}+\|g\|_{\Lambda^\Psi}
$$
When $\Lambda^\Psi=\infty$ and $\Lambda^\Psi=1$ the inequality is immediate, as well as when $\|f\|_{\Lambda^\Psi}=\|g\|_{\Lambda^\Psi}$ $=0$. Suppose that $1<\Lambda^\Psi<\infty$ and $\|f\|_{\Lambda^\Psi}=\alpha \neq 0$ and $\|g\|_{\Lambda^\Psi}=\beta \neq 0$, implying there are functions $f_{0}$ and $g_{0}$ such that $|f|=\alpha f_{0}$ and $|g|=\beta g_{0}$ with $\left\|f_{0}\right\|_{\Lambda^\Psi}=\left\|g_{0}\right\|_{\Lambda^\Psi}=1$.
Now, regard $\lambda=\frac{\alpha}{\alpha+\beta}$ and $1-\lambda=\frac{\beta}{\alpha+\beta}$ note that $0<\lambda<1$, implying
$$
\begin{aligned}
|f(x)+g(x)|^{\Lambda^\Psi} & \leq(|f(x)|+|g(x)|)^{\Lambda^\Psi} \\
& =\left(\alpha f_{0}(x)+\beta g_{0}(x)\right)^{\Lambda^\Psi} \\
& =\left[(\alpha+\beta) \lambda f_{0}(x)+(\alpha+\beta)(1-\lambda) g_{0}(x)\right]^{\Lambda^\Psi} \\
& =(\alpha+\beta)^{\Lambda^\Psi}\left(\lambda f_{0}(x)+(1-\lambda) g_{0}(x)\right)^{\Lambda^\Psi} \\
& \leq(\alpha+\beta)^{\Lambda^\Psi}\left[\lambda\left(f_{0}(x)\right)^{\Lambda^\Psi}+(1-\lambda)\left(g_{0}(x)\right)^{\Lambda^\Psi}\right]
\end{aligned}
$$
Since $\varphi(t)=t^{\Lambda^\Psi}$ is convex in $[0, \infty)$, integrating in (3.4) we have
$$
\begin{aligned}
\int_{X}|f(x)+g(x)|^{\Lambda^\Psi} \mathrm{~d} \mu & \leq(\alpha+\beta)^{\Lambda^\Psi}\left[\lambda\left\|f_{0}\right\|_{\Lambda^\Psi}^{\Lambda^\Psi}+(1-\lambda)\left\|g_{0}\right\|_{\Lambda^\Psi}^{\Lambda^\Psi}\right] \\
& =(\alpha+\beta)^{\Lambda^\Psi}<\infty
\end{aligned}
$$
i.e., $f+g \in L_{\Lambda^\Psi}(X, \mathscr{A}, \mu)$. Finally,
$$
\|f+g\|_{\Lambda^\Psi}^{\Lambda^\Psi} \leq\left(\|f\|_{\Lambda^\Psi}+\|g\|_{\Lambda^\Psi}\right)^{\Lambda^\Psi}
$$
thus
$$
\|f+g\|_{\Lambda^\Psi} \leq\|f\|_{\Lambda^\Psi}+\|g\|_{\Lambda^\Psi}
$$
which ends the proof.
\end{proof}
We are now in condition to acquaint a norm in the $\theta$-Lebesgue space.
\begin{definition}
The $\theta$-Lebesgue space $\left(L_{\Lambda^\Psi}(X, \mathscr{A}, \mu),\|\cdot\|_{L_{\Lambda^\Psi}}\right)$ is a normed space with the norm
$$
\|f\|_{\Lambda^\Psi}=\|f\|_{L_{\Lambda^\Psi}}:=\left(\int_{X}|f|^{\Lambda^\Psi} \mathrm{~d} \mu\right)^{\frac{1}{\Lambda^\Psi}}
$$
whenever $1 \leq \Lambda^\Psi<+\infty$.
\end{definition}
Now, we investigate under which conditions the product of two functions stays in $L_{1}(X, \mathscr{A}, \mu)$. The following mathematical consequence says that if $f \in L_{\Lambda^\Psi}(X, \mathscr{A}, \mu)$ and $g \in L_{\frac{\Lambda^\Psi}{\Lambda^\Psi - 1}}(X, \mathscr{A}, \mu)$ for $\Lambda^\Psi$ and $\frac{\Lambda^\Psi}{\Lambda^\Psi - 1}$ conjugated numbers, i.e., $\frac{1}{\Lambda^\Psi}+\frac{1}{\frac{\Lambda^\Psi}{\Lambda^\Psi - 1}}=1$, we have that $f g \in L_{1}(X, \mathscr{A}, \mu)$. Prior to the demonstration of this powerful mathematical consequence, we require the following lemma.

\begin{lemma}
Let $1 \leq \Lambda^\Psi<\infty$. Then for nonnegative numbers $a, b$, and $t$ we have
$$
(a+t b)^{\Lambda^\Psi} \geq a^{\Lambda^\Psi}+\Lambda^\Psi t b a^{\Lambda^\Psi-1} .
$$
\end{lemma}
Let us demarcate
$$
\phi(t)=(a+t b)^{\Lambda^\Psi}-a^{\Lambda^\Psi}-\Lambda^\Psi t b a^{\Lambda^\Psi-1} .
$$
Note that $\varphi(0)=0$ and $\phi^{\prime}(t)=b \Lambda^\Psi\left[(a+t b)^{\Lambda^\Psi-1}-a^{\Lambda^\Psi-1}\right] \geq 0$ since $\Lambda^\Psi \geq 1$ and $a, b$, $t$ are nonnegative numbers. Therefore $\varphi$ is increasing in $[0, \infty)$ which gives that is nonnegative for $t \geq 0$. Thus, $\varphi(t) \geq \varphi(0)$ and $(a+t b)^{\Lambda^\Psi} \geq a^{\Lambda^\Psi}+\Lambda^\Psi t b a^{\Lambda^\Psi-1}$.
\begin{theorem}
Let $\Lambda^\Psi$ and $\frac{\Lambda^\Psi}{\Lambda^\Psi - 1}$ be extended nonnegative numbers such that $\frac{1}{\Lambda^\Psi}+\frac{1}{\frac{\Lambda^\Psi}{\Lambda^\Psi - 1}}=1$ and $f \in L_{\Lambda^\Psi}(X, \mathscr{A}, \mu), g \in L_{\frac{\Lambda^\Psi}{\Lambda^\Psi - 1}}(X, \mathscr{A}, \mu)$. Then $f g \in$ $L_{1}(X, \mathscr{A}, \mu)$ and
$$
\int_{X}|f g| \mathrm{d} \mu \leq\|f\|_{\Lambda^\Psi}\|g\|_{\frac{\Lambda^\Psi}{\Lambda^\Psi - 1}}
$$
\end{theorem}
Equality holds if there are constants $a$ and $b$, not simultaneously zero, such that $a|f|^{\Lambda^\Psi}=b|g|^{\frac{\Lambda^\Psi}{\Lambda^\Psi - 1}} \mu$-a.e.
\begin{proof}
First regard $\Lambda^\Psi=1$ and $\frac{\Lambda^\Psi}{\Lambda^\Psi - 1}=\infty$, implying clearly

$$
|g| \leq\|g\|_{\infty} \quad \mu \text {-a.e. }
$$

since $|f| \geq 0$, we have that $|f g| \leq|f|\|g\|_{\infty} \mu$-a.e. therefore

$$
\int_{X}|f g| \mathrm{d} \mu \leq\left(\int_{X}|f| \mathrm{d} \mu\right)\|g\|_{\infty}
$$

thus

$$
\int_{X}|f g| \mathrm{d} \mu \leq\|f\|_{1}\|g\|_{\infty}
$$

Now, suppose that $1<\Lambda^\Psi<\infty, 1<\frac{\Lambda^\Psi}{\Lambda^\Psi - 1}<\infty$ and $f \geq 0, g \geq 0$. demarcate $h(x)=[g(x)]^{\frac{\Lambda^\Psi}{\Lambda^\Psi - 1} / \Lambda^\Psi}$, implying

$$
g(x)=[h(x)]^{\Lambda^\Psi / \frac{\Lambda^\Psi}{\Lambda^\Psi - 1}}=[h(x)]^{\Lambda^\Psi-1}
$$

Using Lemma 3.19 we have

$$
\Lambda^\Psi t f(x) g(x)=\Lambda^\Psi t f(x)[h(x)]^{\Lambda^\Psi-1} \leq(h(x)+t f(x))^{\Lambda^\Psi}-[h(x)]^{\Lambda^\Psi} .
$$

thus,

$$
\Lambda^\Psi t \int_{X} f(x) g(x) \mathrm{d} \mu \leq \int_{X}(h(x)+t f(x))^{\Lambda^\Psi} \mathrm{~d} \mu-\int_{X}[h(x)]^{\Lambda^\Psi} \mathrm{~d} \mu=\|h+t f\|_{\Lambda^\Psi}^{\Lambda^\Psi}-\|h\|_{\Lambda^\Psi}^{\Lambda^\Psi} .
$$

From $\theta$-Minkowski's inequality, we have

$$
\Lambda^\Psi \int_{X} f(x) g(x) \mathrm{d} \mu \leq \frac{\left(\|h\|_{\Lambda^\Psi}+t\|f\|_{\Lambda^\Psi}\right)^{\Lambda^\Psi}-\|h\|_{\Lambda^\Psi}^{\Lambda^\Psi}}{t}
$$

Taking $f(t)=\left(\|h\|_{\Lambda^\Psi}+t\|f\|_{\Lambda^\Psi}\right)^{\Lambda^\Psi}$, we deduce $f(0)=\|h\|_{\Lambda^\Psi}^{\Lambda^\Psi}$. Then

$$
\begin{aligned}
\Lambda^\Psi \int_{X} f g \mathrm{~d} \mu & \leq \lim _{t \rightarrow 0} \frac{f(t)-f(0)}{t}=f^{\prime}(0) \\
& =\Lambda^\Psi\left(\|h\|_{\Lambda^\Psi}\right)^{\Lambda^\Psi-1}\|f\|_{\Lambda^\Psi} .
\end{aligned}
$$

Note that

$$
\begin{aligned}
\left(\int_{X}[h(x)]^{\Lambda^\Psi} \mathrm{~d} \mu\right)^{\frac{\Lambda^\Psi-1}{\Lambda^\Psi}} & =\left(\int_{X}[g(x)]^{\frac{\Lambda^\Psi}{\Lambda^\Psi - 1}} \mathrm{~d} \mu\right)^{1-\frac{1}{\Lambda^\Psi}} \\
& =\left(\int_{X}[g(x)]^{\frac{\Lambda^\Psi}{\Lambda^\Psi - 1}} \mathrm{~d} \mu\right)^{\frac{1}{\frac{\Lambda^\Psi}{\Lambda^\Psi - 1}}}
\end{aligned}
$$

implying, $\|h\|_{\Lambda^\Psi}^{\Lambda^\Psi-1}=\|g\|_{\frac{\Lambda^\Psi}{\Lambda^\Psi - 1}}$. Thus

$$
\int_{X} f g \mathrm{~d} \mu \leq\|f\|_{\Lambda^\Psi}\|g\|_{\frac{\Lambda^\Psi}{\Lambda^\Psi - 1}}
$$

Finally, choosing $a=\|g\|_{\frac{\Lambda^\Psi}{\Lambda^\Psi - 1}}^{\frac{\Lambda^\Psi}{\Lambda^\Psi - 1}}$ and $b=\|f\|_{\Lambda^\Psi}^{\Lambda^\Psi}$ such that $a|f|^{\Lambda^\Psi}=b|g|^{\frac{\Lambda^\Psi}{\Lambda^\Psi - 1}}$, implying

$$
|f|=\|f\|_{\Lambda^\Psi} \frac{|g|^{\frac{\Lambda^\Psi}{\Lambda^\Psi - 1} / \Lambda^\Psi}}{\|g\|_{\frac{\Lambda^\Psi}{\Lambda^\Psi - 1}}^{\frac{\Lambda^\Psi}{\Lambda^\Psi - 1} / \Lambda^\Psi}}
$$

and integrating we deduce the required result
\end{proof}
We will give another proof of the $\theta$-Hölder inequality using $\theta$-Minkowski's inequality.
We are now  furnishing one more alternative proof of the $\theta$-Hölder inequality.
\begin{proof}
Let $f \in L_{\Lambda^\Psi}$ and $g \in L_{\frac{\Lambda^\Psi}{\Lambda^\Psi - 1}}$. demarcate $F=|f|^{\Lambda^\Psi}$ and $G=|g|^{\frac{\Lambda^\Psi}{\Lambda^\Psi - 1}}$, which conforms that $F^{1 / \Lambda^\Psi} \in L_{\Lambda^\Psi}$ and $G^{1 / \Lambda^\Psi} \in L_{\frac{\Lambda^\Psi}{\Lambda^\Psi - 1}}$. Now we deduce

$$
\begin{aligned}
\left\|\theta F^{1 / \Lambda^\Psi}+(1-\theta) G^{1 / \Lambda^\Psi}\right\|_{L_{\Lambda^\Psi}} & \leq\left\|\theta F^{1 / \Lambda^\Psi}\right\|_{L_{\Lambda^\Psi}}+\left\|(1-\theta) G^{1 / \Lambda^\Psi}\right\|_{L_{\Lambda^\Psi}} \\
& =\theta\left\|F^{1 / \Lambda^\Psi}\right\|_{L_{\Lambda^\Psi}}+(1-\theta)\left\|G^{1 / \Lambda^\Psi}\right\|_{L_{\Lambda^\Psi}}
\end{aligned}
$$
or in integral terms
$$
\begin{aligned}
\int_{X}\left(\theta F^{1 / \Lambda^\Psi}+(1-\theta) G^{1 / \Lambda^\Psi}\right)^{\Lambda^\Psi} \mathrm{~d} \mu \\
\leq\left[\theta\left(\int_{X} F \mathrm{~d} \mu\right)^{1 / \Lambda^\Psi}+(1-\theta)\left(\int_{X} G \mathrm{~d} \mu\right)^{1 / \Lambda^\Psi}\right]^{\Lambda^\Psi} .
\end{aligned}
$$
we can now deduce
$$
\int_{X} F^{\theta} G^{(1-\theta)} \mathrm{d} \mu \leq\left(\int_{X} F \mathrm{~d} \mu\right)^{\theta} \cdot\left(\int_{X} G \mathrm{~d} \mu\right)^{(1-\theta)}
$$

which is exactly

$$
\int_{X}|f|^{\Lambda^\Psi \theta} \cdot|g|^{\frac{\Lambda^\Psi}{\Lambda^\Psi - 1}(1-\theta)} \mathrm{d} \mu \leq\left(\int_{X}|f|^{\Lambda^\Psi} \mathrm{~d} \mu\right)^{\theta}\left(\int_{X}|g|^{\frac{\Lambda^\Psi}{\Lambda^\Psi - 1}} \mathrm{~d} \mu\right)^{(1-\theta)}
$$

Taking $\theta=1 / \Lambda^\Psi$ we deduce $\theta$-Hölder's inequality.
\end{proof}
\begin{theorem}
Let $\Lambda^\Psi_{k}>1$ be such that $\sum_{k=1}^{n} \frac{1}{\Lambda^\Psi_{k}}=1$. If $f_{k} \in L_{\Lambda^\Psi_{k}}(X, \mathscr{A}, \mu)$, for all $k=1,2, \ldots, n$, than we have that $f_{1} \times f_{2} \cdots \times f_{n} \in L_{1}(X, \mathscr{A}, \mu)$ and

$$
\int_{X}\left|\prod_{k=1}^{n} f_{k}\right| \mathrm{d} \mu \leq \prod_{k=1}^{n}\|f\|_{\Lambda^\Psi_{k}}
$$
\end{theorem}
\begin{proof}
 We give the proof for $n=3$. Let $\frac{1}{\Lambda^\Psi}+\frac{1}{\frac{\Lambda^\Psi}{\Lambda^\Psi - 1}}+\frac{1}{s}=1$ and take $\frac{1}{\Lambda^\Psi}$, implying $\frac{s}{\Lambda^\Psi}+\frac{s}{\frac{\Lambda^\Psi}{\Lambda^\Psi - 1}}=1$, implying that $\frac{1}{s}+\frac{1}{r}=1$. We want to demonstrate that $f g \in L_{s}(X, \mathscr{A}, \mu)$. Indeed,

$$
\int_{X}|f g|^{s} \mathrm{~d} \mu \leq\left(\int_{X}|f|^{s \Lambda^\Psi / s} \mathrm{~d} \mu\right)^{s / \Lambda^\Psi}\left(\int_{X}|g|^{s \frac{\Lambda^\Psi}{\Lambda^\Psi - 1} / s} \mathrm{~d} \mu\right)^{s / \frac{\Lambda^\Psi}{\Lambda^\Psi - 1}}
$$

i.e.

$$
\left(\int_{X}|f g|^{s} \mathrm{~d} \mu\right)^{1 / s} \leq\|f\|_{\Lambda^\Psi}\|g\|_{\frac{\Lambda^\Psi}{\Lambda^\Psi - 1}}
$$

therefore $f g \in L_{s}(X, \mathscr{A}, \mu)$. Finally, we deduce

$$
\int_{X}|f g h| \mathrm{d} \mu \leq\left(\int_{X}|f g|^{s} \mathrm{~d} \mu\right)^{1 / s}\left(\int_{X}|h|^{r} \mathrm{~d} \mu\right)^{1 / r} \leq\|f\|_{\Lambda^\Psi}\|g\|_{\frac{\Lambda^\Psi}{\Lambda^\Psi - 1}}\|h\|_{r} .
$$

The general case follows by similar arguments.
\end{proof}
Example. As an application of $\theta$-Holdër's inequality, we demonstrate that the Gamma function $\Gamma:(0, \infty) \rightarrow \mathbb{R}$ given by

$$
\Gamma(\Lambda^\Psi)=\int_{0}^{\infty} \mathrm{e}^{-t} t^{\Lambda^\Psi-1} \mathrm{~d} t
$$

is a log-convex function, i.e., it satisfies $\varphi(\lambda x+(1-\lambda) y) \leq \varphi(x)^{\lambda} \varphi(y)^{1-\lambda}$ for $0<$ $\lambda<1$ and $x, y$ in the domain of $\varphi$. Let $x, y \in(0, \infty), 0<\lambda<1, \Lambda^\Psi=1 / \lambda$ and $\frac{\Lambda^\Psi}{\Lambda^\Psi - 1}=$ $1 /(1-\lambda)$. Let us take

$$
f(t)=t^{\frac{x-1}{\Lambda^\Psi}} \mathrm{e}^{-\frac{t}{\Lambda^\Psi}}, \quad g(t)=t^{\frac{y-1}{\frac{\Lambda^\Psi}{\Lambda^\Psi - 1}}} \mathrm{e}^{-\frac{t}{\frac{\Lambda^\Psi}{\Lambda^\Psi - 1}}}
$$

and now by Holdër's inequality we deduce

$$
\int_{\varepsilon}^{N} f(t) g(t) \mathrm{d} t \leq\left(\int_{\varepsilon}^{N} f(t)^{\Lambda^\Psi} \mathrm{~d} t\right)^{\frac{1}{\Lambda^\Psi}}\left(\int_{\varepsilon}^{N} g(t)^{\frac{\Lambda^\Psi}{\Lambda^\Psi - 1}} \mathrm{~d} t\right)^{\frac{1}{\frac{\Lambda^\Psi}{\Lambda^\Psi - 1}}}
$$

Now taking $\varepsilon \rightarrow 0$ and $N \rightarrow \infty$ we deduce

$$
\Gamma\left(\frac{x}{\Lambda^\Psi}+\frac{y}{\frac{\Lambda^\Psi}{\Lambda^\Psi - 1}}\right) \leq \Gamma(x)^{\frac{1}{\Lambda^\Psi}} \Gamma(y)^{\frac{1}{\frac{\Lambda^\Psi}{\Lambda^\Psi - 1}}}
$$

since $f(t) g(t)=t^{x / \Lambda^\Psi+y / \frac{\Lambda^\Psi}{\Lambda^\Psi - 1}-1} \mathrm{e}^{-t}, f(t)^{\Lambda^\Psi}=t^{x-1} \mathrm{e}^{-t}$ and $g(t)^{\frac{\Lambda^\Psi}{\Lambda^\Psi - 1}}=t^{y-1} \mathrm{e}^{-t}$.

There is a reverse $\theta$-Hölder type inequality where we can deduce information from one of the integrand functions knowing an a priori uniformly bound with respect to the other integrand function.

The following mathematical consequence furnishes us with another characterization of the norm $\|\cdot\|_{\Lambda^\Psi}$.

\begin{theorem}
Let $f \in L_{\Lambda^\Psi}(X, \mathscr{A}, \mu)$ with $1 \leq \Lambda^\Psi<\infty$, implying

$$
\|f\|_{\Lambda^\Psi}=\|f\|_{L_{\Lambda^\Psi}}=\sup _{g \in L_{\frac{\Lambda^\Psi}{\Lambda^\Psi - 1}}(X, \mathscr{A}, \mu)}\left\{\|f g\|_{1}\|g\|_{\frac{\Lambda^\Psi}{\Lambda^\Psi - 1}}^{-1}: g \neq 0, \frac{1}{\Lambda^\Psi}+\frac{1}{\frac{\Lambda^\Psi}{\Lambda^\Psi - 1}}=1\right\} .
$$
\end{theorem}
\begin{proof}
Using $\theta$-Hölder's inequality we have
$$
\|f g\|_{1}=\int_{X}|f g| \mathrm{d} \mu \leq\|f\|_{\Lambda^\Psi}\|g\|_{\frac{\Lambda^\Psi}{\Lambda^\Psi - 1}}
$$
implying
$$
\|f g\|_{1}\|g\|_{\frac{\Lambda^\Psi}{\Lambda^\Psi - 1}}^{-1} \leq\|f\|_{\Lambda^\Psi}
$$

for $g \neq 0$, which implies

$$
\sup _{g \in L_{\frac{\Lambda^\Psi}{\Lambda^\Psi - 1}}(X, \mathscr{A}, \mu)}\left\{\|f g\|_{1}\|g\|_{\frac{\Lambda^\Psi}{\Lambda^\Psi - 1}}^{-1}: g \neq 0, \frac{1}{\Lambda^\Psi}+\frac{1}{\frac{\Lambda^\Psi}{\Lambda^\Psi - 1}}=1\right\} \leq\|f\|_{\Lambda^\Psi}
$$

Moreover, suppose $f \neq 0$ and $g=c|f|^{\Lambda^\Psi-1}$ ( $c$ constant), implying

$$
|f g|=c|f|^{\Lambda^\Psi}
$$

thus

$$
\|f g\|_{1}=c\|f\|_{\Lambda^\Psi}^{\Lambda^\Psi}
$$

If we choose $c=\|f\|_{\Lambda^\Psi}^{1-\Lambda^\Psi}$ we deduce

$$
\|f g\|_{1}=\|f\|_{\Lambda^\Psi}^{1-\Lambda^\Psi}\|f\|_{\Lambda^\Psi}^{\Lambda^\Psi}=\|f\|_{\Lambda^\Psi} .
$$

Now

$$
|g|^{\frac{\Lambda^\Psi}{\Lambda^\Psi - 1}}=c^{\frac{\Lambda^\Psi}{\Lambda^\Psi - 1}}|f|^{\frac{\Lambda^\Psi}{\Lambda^\Psi - 1}(\Lambda^\Psi-1)}
$$

and integrating both sides give us

$$
\|g\|_{\frac{\Lambda^\Psi}{\Lambda^\Psi - 1}}=c\left(\int_{X}|f|^{\Lambda^\Psi} \mathrm{~d} \mu\right)^{1 / \frac{\Lambda^\Psi}{\Lambda^\Psi - 1}}=\|f\|_{\Lambda^\Psi}^{1-\Lambda^\Psi}\|f\|_{\Lambda^\Psi}^{\Lambda^\Psi / \frac{\Lambda^\Psi}{\Lambda^\Psi - 1}}=\|f\|_{\Lambda^\Psi}^{1-\Lambda^\Psi}\|f\|_{\Lambda^\Psi}^{\Lambda^\Psi-1}=1
$$

since $f \neq 0$, implying $\|g\|_{\frac{\Lambda^\Psi}{\Lambda^\Psi - 1}}^{-1}=1$.

Thus, we can express above as

$$
\|f\|_{\Lambda^\Psi}=\|f g\|_{1}\|g\|_{\frac{\Lambda^\Psi}{\Lambda^\Psi - 1}}^{-1} \leq \sup _{g \in L_{\frac{\Lambda^\Psi}{\Lambda^\Psi - 1}}(X, \mathscr{A}, \mu)}\left\{\|f g\|_{1}\|g\|_{\frac{\Lambda^\Psi}{\Lambda^\Psi - 1}}^{-1}: g \neq 0, \frac{1}{\Lambda^\Psi}+\frac{1}{\frac{\Lambda^\Psi}{\Lambda^\Psi - 1}}=1\right\} .
$$
We can now deduce the mathematical consequence.
\end{proof}
We now give a mathematical consequence, sometimes called the $\theta$-integral Minkowski inequality or even generalized Minkowski inequality, which is a corollary of the characterization of the $\theta$-Lebesgue norm.
\begin{theorem}
($\theta$-Integral Minkowski inequality). Let $\left(X, \mathscr{A}_{1}, \mu\right)$ and $\left(Y, \mathscr{A}_{2}, \mu\right)$ be $\sigma$-finite measure spaces. Suppose that $f$ is a measurable $\mathscr{A}_{1} \times \mathscr{A}_{2}$ function and $f(\cdot, y) \in L_{\Lambda^\Psi}(\mu)$ for all $y \in Y$. Then for $1 \leq \Lambda^\Psi<\infty$ we have

$$
\left\|\int_{Y} f(\cdot, y) \mathrm{d} y\right\|_{L_{\Lambda^\Psi}(X)} \leq \int_{Y}\|f(\cdot, y)\|_{L_{\Lambda^\Psi}(X)} \mathrm{d} y
$$

where the dot means that the norm is taken with respect to the first variable.
\end{theorem}
\begin{proof}
Let us demarcate $a(x)=\int_{Y} f(x, y) \mathrm{d} y$. We have
$$
\begin{aligned}
\|a\|_{L_{\Lambda^\Psi}(X)} & =\sup _{\substack{g \in L^{\frac{\Lambda^\Psi}{\Lambda^\Psi - 1}(X)} \\
\|g\|_{\frac{\Lambda^\Psi}{\Lambda^\Psi - 1}}=1}} \int_{X}|a(x) g(x)| \mathrm{d} x \\
& =\sup _{\substack{g \in L^{\frac{\Lambda^\Psi}{\Lambda^\Psi - 1}}(X) \\
\|g\|_{\frac{\Lambda^\Psi}{\Lambda^\Psi - 1}}=1}} \int_{X}\left|\int_{Y} f(x, y) g(x) \mathrm{d} y\right| \mathrm{d} x \\
& \leq \sup _{\substack{g \in L^{\frac{\Lambda^\Psi}{\Lambda^\Psi - 1}}(X) \\
\|g\|_{\frac{\Lambda^\Psi}{\Lambda^\Psi - 1}}=1}} \int_{Y} \int_{X}|f(x, y) g(x)| \mathrm{d} x \mathrm{~d} y \\
& =\int_{Y}\|f(\cdot, y)\|_{L_{\Lambda^\Psi}(X)} \mathrm{d} y
\end{aligned}
$$

where the first and last equalities are just consequences of the characterization given above for the norm of an $L_{\Lambda^\Psi}$ function whereas the inequality is a consequence of Fubini-Tonelli theorem and the inequality $\left|\int f\right| \leq \int|f|$.
\end{proof}
\begin{theorem}
Let $k \in L^{1}\left(\mathbb{R}^{n}\right)$ and $f \in L_{\Lambda^\Psi}\left(\mathbb{R}^{n}\right)$. Then
$$
\left\|\int_{\mathbb{R}^{n}} k(x-t) f(t) \mathrm{d} t\right\|_{L_{\Lambda^\Psi}\left(\mathbb{R}^{n}\right)} \leq\|k\|_{L^{1}\left(\mathbb{R}^{n}\right)}\|f\|_{L_{\Lambda^\Psi}\left(\mathbb{R}^{n}\right)}
$$
\end{theorem}
\begin{proof}
By a linear change of variables we have
$$
\int_{\mathbb{R}^{n}} k(x-t) f(t) \mathrm{d} t=\int_{\mathbb{R}^{n}} k(t) f(x-t) \mathrm{d} t
$$

which gives

$$
\begin{aligned}
\left\|\int_{\mathbb{R}^{n}} k(x-t) f(t) \mathrm{d} t\right\|_{L_{\Lambda^\Psi}\left(\mathbb{R}^{n}\right)} & =\left\|\int_{\mathbb{R}^{n}} k(t) f(x-t) \mathrm{d} t\right\|_{L_{\Lambda^\Psi}\left(\mathbb{R}^{n}\right)} \\
& \leq \int_{\mathbb{R}^{n}}\|k(t) f(\cdot-t)\|_{L_{\Lambda^\Psi}\left(\mathbb{R}^{n}\right)} \mathrm{d} t \\
& \leq \int_{\mathbb{R}^{n}} \mid k(t)\|\| f(\cdot-t) \|_{L_{\Lambda^\Psi}\left(\mathbb{R}^{n}\right)} \mathrm{d} t \\
& \leq \int_{\mathbb{R}^{n}} \mid k(t)\|f\|_{L_{\Lambda^\Psi}\left(\mathbb{R}^{n}\right)} \mathrm{d} t \\
& =\|k\|_{L^{1}\left(\mathbb{R}^{n}\right)}\|f\|_{L_{\Lambda^\Psi}\left(\mathbb{R}^{n}\right)} .
\end{aligned}
$$
The first inequality is a consequence of the $\theta$-integral Minkowski inequality, the third inequality is due to the fact that $\|f(\cdot-t)\|_{L_{\Lambda^\Psi}\left(\mathbb{R}^{n}\right)}=\|f\|_{L_{\Lambda^\Psi}\left(\mathbb{R}^{n}\right)}$
\end{proof}
We now demonstrate that the $L_{\infty}$-norm can be obtained from the $L_{\Lambda^\Psi}$-norm by a limiting process.
\begin{theorem}
Let $f \in L_{1}(X, \mathscr{A}, \mu) \cap L_{\infty}(X, \mathscr{A}, \mu)$. Then

(a) $f \in L_{\Lambda^\Psi}(X, \mathscr{A}, \mu)$ for $1<\Lambda^\Psi<\infty$.

(b) $\lim _{\Lambda^\Psi \rightarrow \infty}\|f\|_{\Lambda^\Psi}=\|f\|_{\infty}$.
\end{theorem}
\begin{proof}
Let $f \in L_{1}(X, \mathscr{A}, \mu) \cap L_{\infty}(X, \mathscr{A}, \mu)$. Since $|f| \leq\|f\|_{\infty} \mu$-a.e., implying we have $|f|^{\Lambda^\Psi-1} \leq\|f\|_{\infty}^{\Lambda^\Psi-1}$ therefore $|f|^{\Lambda^\Psi} \leq\|f\|_{\infty}^{\Lambda^\Psi-1}|f|$ whence

$$
\|f\|_{\Lambda^\Psi} \leq\|f\|_{\infty}^{1-\frac{1}{\Lambda^\Psi}}\|f\|_{1}^{\frac{1}{\Lambda^\Psi}}
$$

i.e., $f \in L_{\Lambda^\Psi}(X, \mathscr{A}, \mu)$.

(b) Now we have

$$
\limsup _{\Lambda^\Psi \rightarrow \infty}\|f\|_{\Lambda^\Psi} \leq\|f\|_{\infty} .
$$

On the other hand, let $0<\varepsilon<\frac{1}{2}\|f\|_{\infty}$ and

$$
A=\left\{x \in X:|f(x)|>\|f\|_{\infty}-\varepsilon\right\}
$$

note that $\mu(A)>0$, implying

$$
\int_{X}|f|^{\Lambda^\Psi} \mathrm{~d} \mu \geq \int_{A}|f|^{\Lambda^\Psi} \mathrm{~d} \mu \geq\left(\|f\|_{\infty}-\varepsilon\right)^{\Lambda^\Psi} \mu(A)
$$

implying

$$
\liminf _{\Lambda^\Psi \rightarrow \infty}\|f\|_{\Lambda^\Psi} \geq\left(\|f\|_{\infty}-\varepsilon\right) \liminf _{\Lambda^\Psi \rightarrow \infty}[\mu(A)]^{\frac{1}{\Lambda^\Psi}}
$$

since $\varepsilon$ is arbitrary, we deduce

$$
\liminf _{\Lambda^\Psi \rightarrow \infty} \geq\|f\|_{\infty}
$$

combining the above inequalities, we deduce

$$
\|f\|_{\infty} \leq \liminf _{\Lambda^\Psi \rightarrow \infty}\|f\|_{\Lambda^\Psi} \leq \limsup _{\Lambda^\Psi \rightarrow \infty}\|f\|_{\Lambda^\Psi} \leq\|f\|_{\infty} .
$$

So $\lim _{\Lambda^\Psi \rightarrow \infty}\|f\|_{\Lambda^\Psi}=\|f\|_{\infty}$.
\end{proof}


\begin{thebibliography}{99}

\bibitem{IntroLebesgueSpaces2016}
\textit{An Introductory Course in Lebesgue Spaces},
CMS Books in Mathematics, ISBN: 978-3-319-30032-0, 2016.

\bibitem{Pisier2012}
G. Pisier,
\textit{Grothendieck’s Theorem, past and present},
2012,
[Online; accessed 10-February-2015],
{https://ar5iv.org/abs/1101.4195}.

\bibitem{Israfilov2016}
D. M. Israfilov, and A. Testici, 
\textit{Approximation in weighted generalized grand Lebesgue spaces}, 
Colloquium Mathematicae, vol. 143, no. 1, pp. 113-126, 2016.

\bibitem{CaponeFiorenza2005}
C. Capone, and A. Fiorenza,
\textit{On small Lebesgue spaces},
Journal of Function Spaces and Applications, vol. 3, no. 1, pp. 73-89, 2005.

\bibitem{HarjulehtoHasto2006}
P. Harjulehto, P. Hästö, and M. Pere,
\textit{Variable exponent Sobolev spaces on metric measure spaces},
Funct. Approx. Comment. Math., vol. 36, pp. 79--94, 2006.

\bibitem{OstrovskySirota2015}
E. Ostrovsky and L. Sirota,
\textit{Multidimensional Lusin-type inequalities for Grand Lebesgue Spaces},
2015,
{https://archive.org/details/arxiv-1502.03145}.

\bibitem{Diening2011}
L. Diening, P. Harjulehto, P. Hästö, and M. Ruzicka,
\textit{Lebesgue and Sobolev Spaces with Variable Exponents},
Lecture Notes in Mathematics, vol. 2017, Springer-Verlag, January 2011,
DOI: 10.1007/978-3-642-18363-8,
ISBN: 978-3-642-18362-1.
\bibitem{Blei2001}
R. Blei,
\textit{Analysis in Integer and Fractional Dimensions},
Cambridge Studies in Advanced Mathematics, vol. 71, Cambridge University Press, 2001,
Online ISBN: 9780511543012,
{https://doi.org/10.1017/CBO9780511543012}.
\end{thebibliography}
\end{document}